%% file: HighJumpForJournal.tex
\RequirePackage{fix-cm}
\documentclass[smallextended,numbook,envcountsame]{svjour3}       %
\smartqed  
\usepackage{amssymb, amsmath}
\input MathMacrosJDHwithoutTheorems
\input MathMacrosNLP

\usepackage{diagrams}\diagramstyle[tight,centredisplay
,dpi=600,noPostScript,heads=LaTeX]
\usepackage{rotating}
\newcommand{\supjf} { \sup\set{j(f)(\kappa) \st f: \kappa \to
\kappa}} 
\newcommand{\supname}{clearance}
\newcommand{\hj} {high-jump} 
\newcommand{\shj}{super-high-jump}
\newcommand{\ahj}{almost-high-jump}
\newcommand{\sahj}{super-almost-high-jump}
\newcommand{\hjp}{high jump}
\newcommand{\ahjp}{almost high jump}
\newcommand{\shjp}{super high jump}
\newcommand{\sahjp}{super almost high jump}
\newcommand{\hjuecp}{\hjp\ with unbounded excess closure}
\newcommand{\uec}{unbounded excess closure}

\newcommand{\hjot}[1]{high-jump order $#1$}

\newcommand{\figlist}[1]{ 
\addtocounter{figure}{1}
\textsf{\normalsize Figure \thefigure: #1} }

\newarrow{Con} . . . . >
\newarrow{ImpAndCon} ===={=>}
\newarrow{ConAndImp} ===={=>}
\newarrow{Imp} ---->
\newarrow{Iff}<--->
\newarrow{Line} -----
\newarrow{Dot} . . . . .
\newarrow{Dash}{}{dash}{dash}{dash}{}
\newarrow{Bold}====>

\begin{document}

\title{The large cardinals between supercompact and almost-huge}
\author{Norman Lewis Perlmutter}
\institute{Norman Lewis Perlmutter \at
              Florida Atlantic University\\
              Dept. of Mathematical Sciences, SE 234\\
              Florida Atlantic University\\
              777 Glades Road\\
              Boca Raton, FL 33431 \\
              Tel.: +1-561-297-3340 \\
              Fax: +1-561-297-2436\\
              \email{nperlmutter@gc.cuny.edu}
                         \\
}

\date{Received: date / Accepted: date}

\maketitle

\begin{abstract}
I analyze the hierarchy of large cardinals 
between a supercompact cardinal and an almost-huge cardinal. Many of these 
cardinals are defined by modifying the definition of a \hj\ cardinal. A 
high-jump cardinal is defined as the critical point of an elementary embedding 
$j: V \to M$ 
such that $M$ is closed under sequences of length 
$\sup\set{j(f)(\kappa) \st f: 
\kappa \to \kappa}$. Some of the other cardinals analyzed include the 
super-high-jump 
cardinals, almost-high-jump cardinals, 
Shelah-for-supercompactness cardinals, Woodin-for-supercompactness cardinals, 
\Vopenka\ cardinals, hypercompact cardinals, and enhanced supercompact 
cardinals. I organize these cardinals in terms of consistency 
strength 
and implicational strength. I also analyze the superstrong cardinals, which are 
weaker than 
supercompact cardinals but are related to high-jump cardinals. Two of my most 
important results are as follows.
\begin{itemize}
\item
\Vopenka\ cardinals are the same as  
Woodin-for-supercompactness 
cardinals. \item There are no excessively hypercompact cardinals. \end{itemize} 
Furthermore, I prove 
some results relating high-jump cardinals to forcing, as well as analyzing Laver 
functions for super-high-jump cardinals.  
\end{abstract}

\section{Introduction}
The main purpose of this paper is to examine the 
consistency and implicational strengths of several 
large cardinals falling between supercompact and 
almost-huge cardinals. Many of these cardinals are variants of 
the \hj\ cardinals, which are described in definition \ref{definition.HJ}. I
will also investigate superstrong cardinals, which are weaker than supercompact 
cardinals but are closely related to \hj\ cardinals. Many 
of the cardinals that I 
will 
discuss have been used by Apter, Hamkins, and Sargsyan to prove 
several results about universal indestructibility in
\cite{ApterHamkins99:UniversalIndestructibility}, 
\cite{ApterSargsyan2007:ReductionInCons(WoodinSC)}, 
\cite{Apter2008:Reducing(EnhancedSC)},
\cite{ApterSargsyan2010:UniversalIndestructibilty} and
\cite{Apter2011:SomeApplicationsMethod(Hypercompact)}. This paper is adapted 
from the second chapter of my doctoral dissertation, 
\cite{Perlmutter2013:Dissertation}.

Perhaps the most 
interesting result in this paper is the 
main 
result of section \ref{section.WoodinSC=Vopenka}, that a 
Woodin-for-supercompactness cardinal is equivalent to a \Vopenka\ cardinal. 
Another 
noteworthy result is that there are no excessively hypercompact cardinals, 
which is proven in section \ref{section.ExcessivelyHC}.

Recall that an almost-huge cardinal $\kappa$ is 
characterized by an elementary embedding $j: V \to 
M$ with critical point $\kappa$
such that $M$ is closed under $<\!\!j(\kappa)$-sequences 
in $V$.\footnote{When I speak of an elementary embedding, I always intend to 
denote an 
elementary 
embedding with a critical point between transitive 
proper class models of 
$\ZFC$, unless otherwise 
stated.}
Many 
of the large cardinals that I will discuss here 
are natural weakenings of an 
almost-huge cardinal, formed by reducing the level 
of 
closure of the target model. Indeed, in my study of these cardinals, a key 
methodology is to define new large cardinals by weakening, strengthening, or 
otherwise 
modifying 
existing large cardinal definitions. 
Often, the weaker large cardinals will still be sufficient for proving many of 
the same results as the stronger large cardinals. 
Eventually, by repeatedly weakening definitions, one hopes to obtain an 
equiconsistency, as is done in 
\cite{ApterSargsyan2010:UniversalIndestructibilty}. However, in this paper, I 
focus on the large cardinals themselves rather than their applications.

The chart at the end of the introduction summarizes the relationships between 
the large 
cardinals that I discuss in this paper. Most of the remaining sections will be 
dedicated to proving these relationships. The arrows on the 
chart represent relationships between the cardinals, as indicated in the key. A 
solid arrow from $A$ to $B$ indicates a direct implication: every cardinal with 
property $A$ has property $B$. A dotted arrow means that a cardinal of type $A$ 
is \emph{strictly} stronger in consistency than a cardinal of type $B$. That is 
to say, if there is a cardinal of type A, then it is consistent with \ZFC\ that 
there is a 
cardinal of type B. A double arrow indicates that both of these relationships 
hold. 
The arrows are labeled with theorem numbers referring to the theorems, 
propositions, and corollaries in which 
the corresponding results are proven.
Dashed arrows are labeled with two numbers: one for a theorem demonstrating 
the 
consistency implication and one for a theorem  
demonstrating the failure of the direct implication.

Throughout the paper, I will use the following seed theory notation, which has 
been 
popularized by Hamkins, to refer to factor embeddings and the related measures.

\begin{definition}
Let $j: V \to M$ be an elementary embedding with critical point $\kappa$, and 
let $\lambda$ be a cardinal greater than $\kappa$. Let $U$ be a normal fine 
measure on $P_\kappa \lambda$ given by $A \in U \iff j \image \lambda \in j(A)$. 
Then $U$ is the normal fine measure on $P_\kappa \lambda$ \textbf{induced via 
$j$ 
by the seed $j \image \lambda$}, and the ultrapower embedding generated by $U$ 
is 
the 
$\lambda$-supercompactness factor embedding of $j$ induced by the seed $j \image 
\lambda$. 
\end{definition}

The organization of the paper is as follows. The sections 
after section 
\ref{section.HJ} can mostly be read out of order. I have noted the most 
important 
dependencies between 
the sections below. 
In section 
\ref{section.HJ}, I define the clearance of an elementary embedding and use 
this 
property to define the \hj\ cardinals.  I also define and analyze the related 
notions of \ahj\ cardinals, 
Shelah-for-supercompactness cardinals, and \hj\ functions.  In section 
\ref{section.SuperstrongAndClearance}, which depends on 
section \ref{section.HJ}, I analyze properties 
of the clearance of an embedding and prove theorems tying together the ideas of 
the clearance of an embedding, the \ahj\ cardinals, 
and the superstrong cardinals. The next few sections are arranged mostly by 
decreasing strength of the large cardinal notions studied. In section 
\ref{section.CardinalStrengths}, which depends on section \ref{section.HJ} and 
on lemma \ref{lemma.VThetaElemMjkappa}, I define and analyze several large 
cardinals 
above a \Vopenka\ cardinal and below an almost-huge cardinal. In section 
\ref{section.WoodinSC=Vopenka}, I  define the \Vopenka\ and 
Woodin-for-supercompactness cardinals and prove that they are equivalent.
In 
section \ref{section.ExcessivelyHC}, I define the hypercompact cardinals and 
the 
excessively hypercompact cardinals, and I show that the existence of an 
excessively hypercompact cardinal is inconsistent with \ZFC. In section 
\ref{section.EnhancedSC}, I define the enhanced supercompact cardinals and 
analyze their place in the large cardinal hierarchy. In section 
\ref{section.forcing}, which depends on section \ref{section.HJ}, I consider 
the 
relationship between high-jump cardinals 
and forcing. In section \ref{section.Laver}, which depends on section
\ref{section.HJ} and on lemma \ref{lemma.VThetaElemMjkappa}, I develop 
analogues of Laver 
functions for \hj\ cardinals and related cardinals. In section 
\ref{section.FurtherDefinitionIdeas}, I review open 
problems and directions for further research.

I use the label \emph{theorem} to denote very important 
results. The results labeled as \emph{propositions} 
vary in their mathematical depth. Some of them might more appropriately be 
considered as examples or observations.

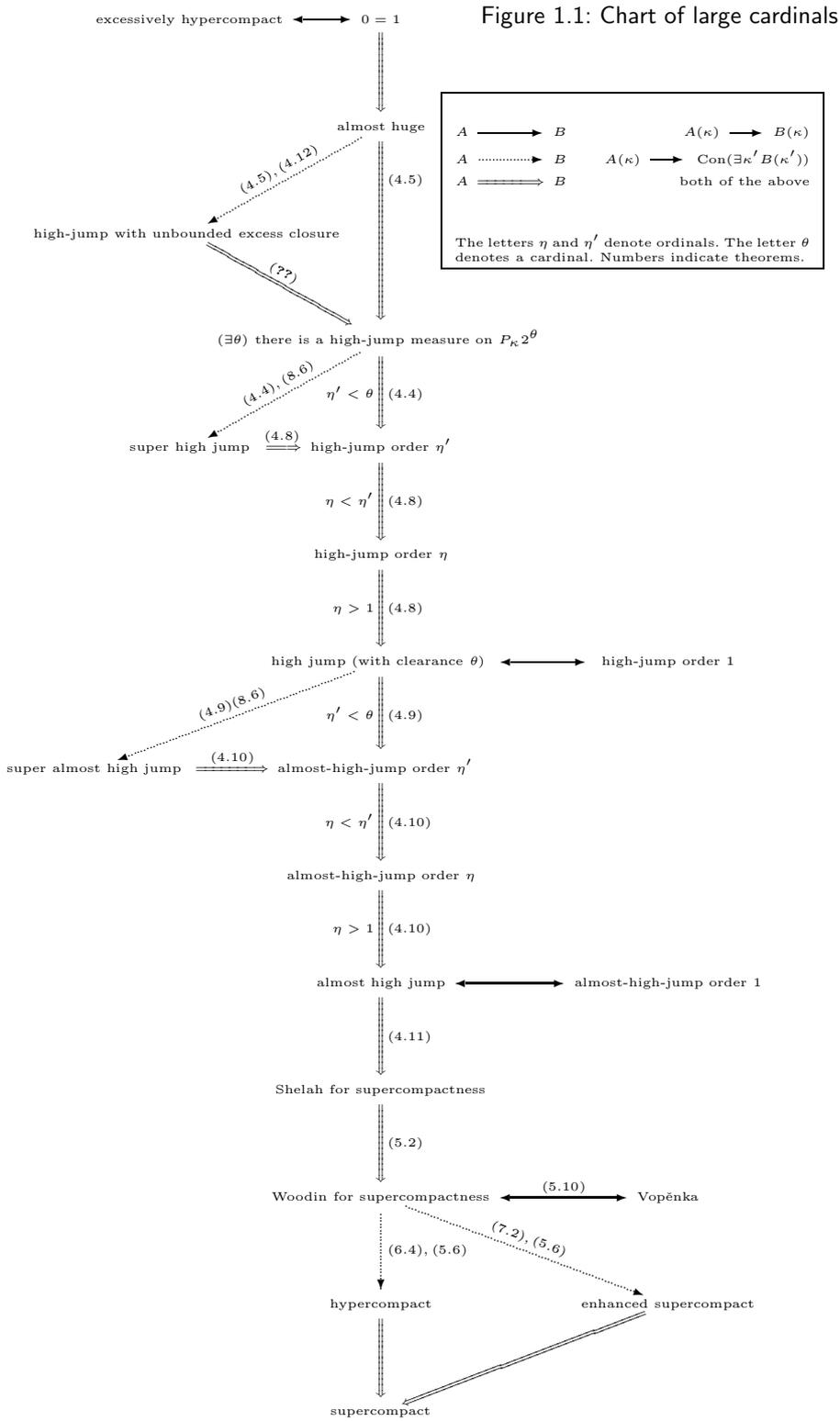
\addcontentsline{toc}{section}{\numberline{}Chart of large cardinal 
relations}
\hspace*{5mm}
{\tiny 
\hspace*{5mm}
\begin{minipage}{3in}
\begin{diagram}[height=22pt,width=39pt]
& & \text{excessively 
hypercompact}  & \rIff & 0=1         & & &  \hspace*{-7pt} 
\figlist{Chart of 
large 
cardinals} 
&     \\
 & & & &\dConAndImp & & & &      \\
& & & & \text{almost huge} & & & & \\ 
& & & \ldCon<{(\ref{proposition.AlmostHuge>SSHJ}),  
(\ref{proposition.LeastHJnotStr})}  & 
\dConAndImp>{(\ref{proposition.AlmostHuge>SSHJ})} & & & & 
\\
& &\text{\hj\ with unbounded excess closure } & &  & & & &\\
& &  &
\rdConAndImp(2,2)^{(\ref{proposition.VeryBigClosure})}  & & & &
\\
& & & 
& \text{$(\exists \theta)$ there is a \hj\ measure on $P_\kappa 2^\theta$ } & & 
& & \\
& & &  \ldCon(2, 2)<{(\ref{proposition.HJplus}),  
(\ref{theorem.HJButNotSC})} & 
\dConAndImp>{(\ref{proposition.HJplus})}< {\eta' < \theta} & & & & 
\\
& & \text{ \shjp\ } & \rConAndImp^{(\ref{proposition.HJplusOrder})} 
& \text{\hj\ order $\eta'$} & & & 
&\\
& & & & \dConAndImp>{(\ref{proposition.HJplusOrder})}<{\eta < \eta'} & & &  & 
\\
& & & & \text{ \hj\ order $\eta$ } & & & &\\ 
& & & & \dConAndImp<{\eta > 1}>{(\ref{proposition.HJplusOrder})} & & & & \\
& & & & \text{\hjp\ (with clearance $\theta$) } & \rIff & & \text{ \hj\ order 
$1$ 
}  
& \\
& & &\ldCon(3, 2)^{(\ref{proposition.HJ>SAHJ}) 
(\ref{theorem.HJButNotSC})} & 
\dConAndImp>{(\ref{proposition.HJ>SAHJ})}<{\eta' < \theta} & & & & \\ \
&  \text{ \sahjp\ } & & \rConAndImp^{(\ref{corollary.AHJorder})} & 
\text{\ahj\ order $\eta'$ 
}\ & & & & \\
& & & 
&\dConAndImp>{(\ref{corollary.AHJorder})}<{\eta < \eta'}
& 
& & & \\
& & & & \text{\ahj\ order $\eta$} & 
& 
& & \\
& & & & \dImpAndCon<{\eta > 1}>{(\ref{corollary.AHJorder})} & & & & \\
& & & & \text{\ahjp} & \rIff & & \text{\ahj\ order 1}  & \\
& & & 
&\dImpAndCon>{(\ref{proposition.AHJ>SSC})}
 & & & & \\
& & & &\text{Shelah for supercompactness} & & & & \\
& & & &\dImpAndCon>{
(\ref{proposition.ShelahSCStrongerThanWoodinSC})} 
& & & & \\
& & & &\text{Woodin for 
supercompactness} 
& 
\rIff^{(\ref{theorem.WSC=Vopenka})}  & 
&  \text{\Vopenka} & \\
& 
&
& &\dCon>{(\ref{theorem.WSC>WHC}), 
(\ref{proposition.LeastVopenkaNotWC})}& 
\rdCon(3,2)>
{(\ref{proposition.WoodinSC>ESC}), 
(\ref{proposition.LeastVopenkaNotWC})}
& & \\
 & & & &\text{hypercompact} & &  
&\text{enhanced supercompact} & \\
&
& & &\dImpAndCon & & \ldConAndImp(3,2 )
& &\\
& & & & \text{supercompact} & & & &
\end{diagram}
\end{minipage}
\hspace*{-32mm}
\raisebox{75mm}{
\fbox{
\begin{minipage}{2in}
\begin{align*}
&\begin{diagram}[width=20pt]
A & \rImp & B
\end{diagram}
& \quad A(\kappa) \rImp B(\kappa) \\
&\begin{diagram}[width=20pt]
A & \rCon & B
\end{diagram}
& \quad \! A(\kappa) \rImp \Con(\exists \kappa' B(\kappa')) 
\\ 
&\begin{diagram}[width=20pt]
A & \rImpAndCon & B
\end{diagram}
& \text{\quad both of the above} \\
\end{align*}
The letters $\eta$ and $\eta'$ denote ordinals. The letter $\theta$ 
denotes a cardinal. Numbers indicate theorems.
\end{minipage}
}}

}

\pagebreak

\section[High-jump, almost-high-jump, and Shelah-for-s.c. cardinals]{High-jump 
cardinals, \ahj\ cardinals, \\ and Shelah-for-supercompactness 
cardinals}\label{section.HJ}
In this section, I define \hj\ cardinals, \ahj\ cardinals, and 
Shelah-for-supercompactness cardinals.\footnote{In many cases, the 
English usage rule for the punctuation of compound adjectives is to hyphenate 
compound adjectives coming before a noun, but not compound adjectives coming 
after a noun. Hence, I will write that $\kappa$ is a high-jump cardinal, but 
also that the cardinal $\kappa$ 
is high jump.} I also give  
characterizations for these large cardinals in terms of 
ultrafilters and prove a lemma about factor embeddings that will be very useful 
for the rest of the paper.

The clearance of an elementary embedding, defined below in 
definition 
\ref{definition.ClearanceHeightAndClosure}, is a key 
concept for defining several large 
cardinals. The motivation for defining the clearance is for use as  
a weaker substitute for $j(\kappa)$ in large cardinal definitions.

\begin{definition}
\label{definition.ClearanceHeightAndClosure}
Let $j: M \to N$ be an elementary embedding with 
critical point 
$\kappa$. The \textbf{clearance} of $j$ denotes the 
ordinal $$\supjf.$$ 
\end{definition}

The notation \emph{\supname} is borrowed from the 
sport of pole vaulting, where the \supname\ is the 
height of the bar that the pole vaulter must clear. 
A 
high-jump embedding is like a pole vaulter: for a 
cardinal to be high jump, the closure of the embedding 
must 
successfully clear the \supname, as is described precisely in 
the following definition. The term \emph{\hj\ cardinal} comes from
\cite{ApterHamkins99:UniversalIndestructibility}. However, these cardinals were 
previously defined in 
\cite[p.111]{SolovayReinhardtKanamori1978:StrongAxiomsOfInfinity}, where they 
were given the designation $A_4$.

\begin{definition} \label{definition.HJ}
The cardinal $\kappa$ is a \textbf{high jump cardinal}
 if and only if
there exists a cardinal $\theta$ and an elementary embedding $j: V \to M$ 
with critical
point $\kappa$ and clearance $\theta$ such that $M^\theta \of M$.

An embedding witnessing that $\kappa$ is \hjp\ is called a \textbf{\hj\ 
embedding} for $\kappa$. A normal fine measure on some $P_\kappa \lambda$ 
generating 
an ultrapower embedding that is a \hj\ embedding is called a \textbf{\hj\ 
measure}. 
\end{definition}



The clearance of an embedding has strong 
properties, as I will 
show in the next section. In particular, I 
will show in corollary 
\ref{Theorem.M_thetaElemInM_j(kappa)} 
that if 
$\theta$ is the clearance of \emph{any} 
elementary embedding $j: V \to N$ with critical 
point 
$\kappa$, then 
$N_\theta \elem N_{j(\kappa)}$. 

The following lemma provides a criterion for showing that a factor embedding of 
a high-jump embeddings is a high-jump embedding. It will be used many times 
throughout the paper.

\begin{lemma} \label{lemma.FactorEmbeddingJump} 
\label{lemma.FactorEmbeddingEquality} \label{lemma.HJFactor}
Let $j: V \to M$ be a \hj\ embedding for $\kappa$ with clearance $\theta$. Let 
$\lambda \geq \theta$ be a cardinal such 
that $j \image \lambda \in M$. Let $j_0: V \to M_0$ be the factor embedding 
induced via $j$ by the seed $j \image \lambda$. Then $j_0$ is a \hj\ embedding, 
and the clearance of $j_0$ is $\theta$. 
\end{lemma}
\begin{diagram}[w=5em]
 V & \rTo^j
& M &  \\
\dTo<{j_0} & \ruTo>{k} & \\
M_0 & & 
\end{diagram}
\figlist{Factor embeddings of a \hj\ embedding} 

\begin{proof}
Let $f: \kappa \to \kappa$ be any function.
Referring to the diagram above, note that the critical point of $k$ is greater 
than $\lambda$. It follows that 
$$j(f)(\kappa) = (k\circ h)(f)(\kappa) = 
\big(k (h(f))\big)(k(\kappa)) = k(h(f)(\kappa)).$$

The ordinal $j(f)(\kappa)$ is less than $\theta$, since $\theta$ is the 
clearance of $j$. Therefore, again since $\crit(k) > \lambda$, it must be the 
case that $h(f)(\kappa) 
= j(f)(\kappa)$. Since $f$ was arbitrary, the embedding $j_0$ must have the same 
clearance as the embedding $j$, namely $\theta$. Since $\lambda \geq \theta$, it 
follows that $j_0$ is a \hj\ embedding. 
\qed
\end{proof}

%
%

Next, I provide a combinatorial characterization of 
\hj\ measures.
	
\begin{lemma} \label{theorem.HighJumpMeasure} 
Given an ordered pair of cardinals $(\kappa, \theta)$, the following are 
equivalent.

\begin{enumerate}
\item \label{item.HJemb}
There exists a \hj\ embedding $j: V \to M$ with critical point $\kappa$ 
such that $M^\theta \of M$ and the clearance of $j$ is at most $\theta$.

\item \label{item.HJMeasureLosChar}
There exists a 
normal fine measure $U$ on $P_\kappa \theta$ such 
that for every
function $f: \kappa \to \kappa$, the set $\set{A \in 
P_\kappa
\theta \st f(\ot(A \cap \kappa)) < \ot(A)}$ is a 
member of $U$. (The operator $\ot$ denotes order type.)
\end{enumerate}

\end{lemma}
\begin{proof}
The proof consists of a straightforward argument using the \Los Theorem and 
lemma \ref{lemma.FactorEmbeddingEquality}. For the details, see \cite[lemma 
55]{Perlmutter2013:Dissertation}. 
\qed
\end{proof}

Next, I define the \ahj\ cardinals by a slight weakening of the 
closure property used for defining \hj\ cardinals. An \ahj\ 
cardinal is to a \hj\ cardinal as an almost-huge cardinal is to 
a huge cardinal.

\begin{definition} A 
cardinal 
$\kappa$ is \textbf{almost high jump} if 
and only if there exists an elementary embedding $j: V \to M$ with critical 
point $\kappa$ and clearance $\theta$ such that $M^{<\theta} \of M$. Such an 
embedding is called an \ahj\ embedding for $\kappa$.

\end{definition}

Another way to look at the definition of an \ahj\ cardinal is as follows. The 
cardinal $\kappa$ is almost high jump if and only if there exists an elementary 
embedding 
$j: V 
\to M$ with critical point $\kappa$ such that for every function $f: \kappa \to 
\kappa$, the closure property $M^{j(f)(\kappa)} \of M$ holds.
The \ahj\ cardinals have a combinatorial characterization in terms 
of coherent sequences of normal fine measures. See \cite[lemma 
57]{Perlmutter2013:Dissertation} for details.

Weakening the definition of an \ahj\ cardinal to allow for 
distinct embeddings to 
witness closure with respect to distinct functions $f: \kappa 
\to \kappa$
produces the definition of a Shelah-for-supercompactness cardinal.  The 
analogue 
of this definition for strongness (in place of 
supercompactness) was 
originally formulated by Shelah. 

\begin{definition}  A 
cardinal $\kappa$ is \textbf{Shelah for 
supercompactness} 
if and only if for every function $f: \kappa \to 
\kappa$, there is 
an elementary embedding $j: V \to M$ such that 
$M^{j(f)(\kappa)} \of 
M$.
\end{definition}
Note that an \ahj\ cardinal is a uniform version of a Shelah for 
supercompactness 
cardinal --- with an \ahj\ cardinal, one embedding must be the witness for 
every 
$f$ uniformly, whereas with a Shelah-for-supercompactness cardinal, each 
function $f$ may have a separate witnessing embedding.

One might want to define an almost-Shelah-for-supercompactness cardinal by 
tweaking the above definition to require that the 
closure of the target model is 
only $<\!\!j(f)(\kappa)$. However, this definition 
is actually 
equivalent to a Shelah-for-supercompactness cardinal, because of the following 
argument. Let 
$g: \kappa \to 
\kappa$ be given by $g(\alpha) = f(\alpha)^+$. If $j: V \to M$ is an elementary 
embedding with critical point $\kappa$ such that
$M^{<j(g)(\kappa)} \of M$, then $M^{j(f)(\kappa)} \of M$ as well.

In \cite[p.201]{Hamkins98:AsYouLikeIt}, Hamkins defines a \hj\ function as 
follows.  A \textbf{\hj\ function} for a 
(partially) supercompact cardinal 
$\kappa$ is a function $f \from \kappa \to \kappa$ such that $j(f)(\kappa) > 
\lambda$ 
whenever $j$ is a $\lambda$-supercompactness embedding on 
$\kappa$. Hamkins allows for partial functions, but any partial \hj\ function 
can be extended to a total \hj\ function, so I will assume without 
loss of generality that \hj\ functions are total. Furthermore, I will 
extend the definition of a \hj\ function to the vacuous case where $\kappa$ has 
no supercompactness by saying that in this case, there exists a \hj\ function 
for 
$\kappa.$ The following proposition shows that the existence of a \hj\ function 
for a cardinal $\kappa$ is actually an anti-large-cardinal property.

\begin{proposition} \label{proposition.HJfunction}
Let $\kappa$ be a cardinal. Then there exists a \hj\ function for $\kappa$ 
if and only if $\kappa$ is not Shelah for supercompactness.
\end{proposition}

\begin{proof}
The proof follows immediately from the definitions. The cardinal $\kappa$ is 
Shelah for 
supercompactness if and only if
$$(\forall f: \kappa \to \kappa) (\exists  j: V 
\to M \text{ with critical point $\kappa$) such that } M^{j(f)(\kappa)} \of M$$
The logical negation of this statement is
$$(*) \qquad (\exists f: \kappa \to \kappa) (\forall  j: V 
\to M \text{ with critical point $\kappa$) } M^{j(f)(\kappa)} \nsubseteq M $$

The formula (*) asserts that $f$ is a \hj\ function for $\kappa$.\footnote{If 
one requires that a 
$\lambda$-supercompactness embedding be generated by a normal fine measure on 
$P_\kappa\lambda$ rather than simply defining such embeddings by the closure of 
the target model, then a factor embedding argument is required. See 
\cite[proposition 59]{Perlmutter2013:Dissertation} for details.}
\qed
\end{proof}

The Shelah-for-supercompactness cardinals have an 
ultrafilter characterization similar to that for 
high-jump cardinals, given by the following 
corollary to lemma \ref{theorem.HighJumpMeasure}.

\begin{corollary}
A cardinal $\kappa$ is Shelah-for-supercompactness 
if 
and only if 
for every function $f: \kappa \to \kappa$, there is 
a 
cardinal 
$\theta$ and a normal fine measure $U$ on $P_\kappa 
\theta$ such 
that the set $\set{A \in P_\kappa
\theta \st f(\ot(A \cap \kappa)) < \ot(A)}$ is a 
member of $U$.
\end{corollary}

\begin{proof}
The proof is very similar to that of lemma
\ref{theorem.HighJumpMeasure} and is given in \cite[corollary 
60]{Perlmutter2013:Dissertation}.
\qed
\end{proof}

\section[The clearance and superstrongness]{The clearance, superstrongness 
embeddings, and related embeddings} 
\label{section.SuperstrongAndClearance}

In the large cardinal literature, a cardinal $\kappa$ is \textbf{superstrong} 
if 
and only if there 
exists an elementary embedding $j: V \to M$ such that $V_{j(\kappa)} \of M$. 
 A cardinal 
$\kappa$ is 
\textbf{almost huge} if and only if there exists an elementary embedding $j: V 
\to M$ such that $M^{<j(\kappa)} \of M$.
The chart in the introduction shows that an almost-huge 
cardinal is much stronger in consistency strength than a \hj\ cardinal. 
Remarkably, the analogous situation does not hold in the case of strongness. In 
theorem 
\ref{theorem.SuperstrongCharacterization}, I will show that a 
superstrong cardinal is equivalent to a high-jump-for-strongness 
cardinal.
Before proving this result, I will prove some facts about the clearance of 
an embedding and about \ahj\ embeddings. I begin with the following lemma. 

\begin{lemma}\label{lemma.clearanceBasicsGeneralEmbedding} 
\label{lemma.CHnotAchieved}
Let $j: V \to M$ be an elementary embedding with critical point $\kappa$ and 
clearance $\theta$. Then the following conclusions are true.
\begin{itemize} 
\item
There is no function $f : \kappa \to \kappa$ such that 
$j(f)(\kappa) = \theta.$ 
\item
The ordinal $\theta$ is a $\beth$ fixed 
point in $M$, that is to say, $\beth_\theta^M = \theta$.
\item
The inequality 
$\kappa^+ 
\leq \cof(\theta) 
\leq 2^\kappa$ holds in $V$.
\end{itemize}
\end{lemma}

\begin{proof}
To prove the first conclusion, suppose to the contrary that $f$ is a function 
such that $j(f)(\kappa) = 
\theta.$ Let $g: \kappa \to \kappa$ 
be defined by $g(\alpha) = f(\alpha) + 1$. Then $j(g)(\kappa) = \theta +1 > 
\theta$, contradicting the definition of the clearance.

Next, I will show that $\beth_\beta^M< \theta$ for all ordinals $\beta < 
\theta$, so that $\beth_\theta^M = \theta$.
Let $\beta < \theta$. Then there exists a function $f: 
\kappa \to \kappa$ 
such that $j(f)(\kappa) \geq \beta.$ Let the function $g: \kappa \to \kappa$ be 
given by $g(\alpha) = \beth_{f(\alpha)}$. Then $\beth_\beta^M \leq j(g)(\kappa) 
< 
\theta$. It follows that $\theta$ is a $\beth$ fixed point in $M$.

The cofinality of the clearance $\theta$ must be at most $2^\kappa$, because 
the 
clearance is defined as the 
supremum of a set indexed by functions from $\kappa$ to $\kappa$, of which 
there 
are $2^\kappa$ many.

Finally, I show that $\cof(\theta) \geq \kappa^+$ by a 
diagonalization argument. Suppose to the contrary that $\cof(\theta) \leq 
\kappa.$ Then there is a sequence 
$\<f_\alpha>_{\alpha<\kappa}$ of functions on $\kappa$ such that $\theta = 
\sup\set{j(f_\alpha)(\kappa) \st \alpha < \kappa}.$ Define a function $g: 
\kappa 
\to 
\kappa$ diagonalizing over these functions. That is to say, given $\beta< 
\kappa$, let 
$g(\beta) = \sup\set{f_\alpha(\beta) + 1 \st \alpha \leq \beta} .$ Then 
$j(f_\alpha)(\kappa) < j(g)(\kappa) < \theta$ for every $\alpha < \kappa$, 
contradicting the assumption that $\theta = \sup\set{j(f_\alpha)(\kappa) \st 
\alpha < \kappa}$.
\qed
\end{proof}


The next lemma applies the result of lemma \ref{lemma.CHnotAchieved} in the 
case 
that $j$ is an \ahj\ embedding.

\begin{lemma}\label{lemma.ThetaStrongLimit}
\label{lemma.ClearanceFactsAHJEmbedding}
Suppose $j: V \to M$ is an \ahj\ embedding with critical point 
$\kappa$ and clearance $\theta$. Then the following conclusions are true in 
both 
$V$ 
and $M$.

\begin{itemize}
\item
The cardinal $\theta$ is a singular 
$\beth$ fixed point.
\item 
The inequality 
$\kappa^+ 
\leq \cof(\theta) 
\leq 2^\kappa$ holds. 
\item
The cardinal exponentiation identity $\theta^{\kappa} = \theta$ holds.
\end{itemize}
\end{lemma}

\begin{proof}
The proof follows from lemma \ref{lemma.CHnotAchieved}, along with the 
fact that $M$ is sufficiently closed so that it agrees with $V$ on cofinalities 
less than $\theta$ and on cardinal exponentiation below $\theta$.

To show that $\theta^\kappa = \theta$ in both $V$ and $M$, note that $\theta$ 
is 
a strong limit in both $V$ and $M$ and $\cof(\theta) > \kappa$ in both $V$ and 
$M$. The 
fact that $\theta^{\kappa} = \theta$ in both $V$ and $M$ then follows from a 
basic theorem of 
cardinal 
arithmetic (see \cite[theorem 5.20]{Jech:SetTheory3rdEdition}). 
\qed
\end{proof}

With these preliminaries out of the way, I now state the main theorem of 
this section.

\begin{theorem} \label{theorem.SuperstrongCharacterization}

A cardinal $\kappa$ is 
high jump for strongness if and only if $\kappa$ is superstrong. This fact 
follows from the following stronger but more technical result.

Let $\kappa$ be a cardinal.
Let $j: V \to M$ be a high-jump-for-strongness embedding with critical point 
$\kappa$ and 
clearance $\theta$. Then $V_\theta \elem M_{j(\kappa)}$, and $j$ has a factor 
embedding $h: V \to M'$ such that $h$ is a superstrongness embedding with 
critical point $\kappa$ and such tha
$h(\kappa) = \theta$.

\end{theorem}

\begin{proof}
Let $j: V \to M$ be a \hj-for-strongness embedding with 
critical point $\kappa$ 
and 
clearance $\theta$. 
I define the seed hull of 
$\theta$ in $M$, denoted by 
$X_{\theta}$, as follows. 

$$X_{\theta} = \set{j(f)(\alpha) \st \alpha < \theta \text{ and } f \in V 
\text{ is 
a 
function}}.
$$
The seed hull $X_{\theta}$ is an elementary substructure of $M$, and 
setting $M'$ equal to its Mostowski collapse yields the following 
commutative diagram of elementary embeddings of models of set theory, where $k$ 
is the inverse of the collapse map. 
\begin{diagram}[w=5em] 
V & \rTo^j & M\\
\dTo^{h} & \ruTo>k \\
M'
\end{diagram}
\figlist{Factor embeddings of a \hj\-for-strongness embedding}

Next, I will show that the critical point of $k$ is $\theta$, and $k(\theta) = 
j(\kappa)$. Since $k$ is the inverse of the Mostowski collapse of $X_\theta$, it 
suffices to show that the supremum of
the ordinals $\beta$ of $X_{\theta}$ below 
$j(\kappa)$ is $\theta$. Every such ordinal $\beta$ is of the form 
$j(f)(\alpha)$ for 
some ordinal $\alpha < \theta$ and some function $f: \kappa \to \kappa$. Fix 
such an ordinal $\alpha$ and function $f$. Since $\theta$ is 
the clearance of the embedding $j$, it follows that $\alpha < j(g)(\kappa)$ for 
some other function $g:\kappa \to \kappa$. Define yet another function $g': 
\kappa \to 
\kappa$ by $g'(\beta) = \sup\set{f(\gamma) \st \gamma < g(\beta)}$. By the 
elementarity of $j$, and since $\theta$ is the clearance of the embedding $j$, 
it follows that 
\begin{equation}\label{doohickie}
j(g')(\kappa) = \sup\set{j(f)(\gamma) \st 
\gamma < j(g)(\kappa)} < \theta
\end{equation} 
Considering the case $\gamma = \alpha$ in equation \ref{doohickie} above, it 
follows that $j(f)(\alpha) < \theta$. It follows that $\theta$ is the critical 
point of $k$ and that $k(\theta) = j(\kappa),$ as claimed. Since the diagram 
above 
commutes, it further follows that $h(\kappa) = \theta$. 

Next, I claim that $V_\theta \of M'$. Towards the proof of this claim, 
first recall that since $j$ is a \hj-for-strongness embedding, $V_\theta = 
M_\theta$. Next, let $f: \kappa 
\to V_\kappa$ be an 
enumeration of $V_\kappa$ in $V$ such that whenever $\alpha < \kappa$, it 
follows that  $V_\alpha \of f \image \beth_\alpha$. By lemma 
\ref{lemma.CHnotAchieved}, the ordinal $\theta$ is a $\beth$ fixed point in 
$M$, 
and so it follows from the definitions of $f$ and of $X_\theta$ that 
$M_\theta 
\of 
X_\theta$. Furthermore $M_\theta = V_\theta$, so $V_\theta \of X_\theta.$
Since 
$M'$ is the Mostowski collapse of $X_\theta$ in $M$, it follows that $V_\theta 
\of M'$, as claimed. 

Since $V_\theta \of M'$ and $h(\kappa) = \theta$, it follows that $h$ is a 
superstrongness embedding. Furthermore, the embedding $k$ witnesses that 
$V_\theta \elem M_{j(\kappa)}$. 
\qed
\end{proof}

A few easy corollaries to theorem \ref{theorem.SuperstrongCharacterization} 
follow. I will label the first corollary as a lemma, because it is a key fact 
about \ahj\ 
embeddings and will be used in many places in this paper.

\begin{lemma}\label{lemma.VThetaElemMjkappa}
Let $j: V \to M$ be an \ahj\ embedding for $\kappa$ 
with clearance 
 $\theta$. Then $V_\theta \satisfies \ZFC$ and 
$V_\theta \elem 
M_{j(\kappa)}$.
\end{lemma}

\begin{proof}
By lemma \ref{lemma.ThetaStrongLimit}, the clearance $\theta$ of $j$ is a 
$\beth$ fixed point. Therefore, since $j$ is a $\theta$-supercompactness 
embedding, it is also a $\theta$-strongness embedding, and thus a 
\hj-for-strongness embedding. It follows immediately 
from theorem \ref{theorem.SuperstrongCharacterization} that $V_\theta \elem 
M_{j(\kappa)}$. Moreover, since $j(\kappa)$ is inaccessible in $M$, it follows 
that $V_\theta \satisfies \ZFC$.
\qed
\end{proof}

\begin{corollary} 
\label{Theorem.M_thetaElemInM_j(kappa)}
\label{Corollary.N_thetaElemInN_j(kappa)}
Let $j: V \to M$ be an elementary embedding with 
critical point 
$\kappa$ and clearance $\theta$. Then $M_\theta \satisfies \ZFC$ and
$M_\theta \elem 
M_{j(\kappa)}$. 
\end{corollary}
\begin{proof}
The same line of reasoning as in the proof of theorem 
\ref{theorem.SuperstrongCharacterization} shows that $M_\theta \elem 
M_{j(\kappa)}$, even without the assumption that the embedding $j$ has 
additional strength.
\qed
\end{proof}

\begin{corollary} \label{corollary.SuperstrongFactorofHJ}
Every \ahj\ embedding has a superstrongness factor embedding.
\end{corollary}
\begin{proof}
This follows immediately from theorem 
\ref{theorem.SuperstrongCharacterization}, 
since every \ahj\ embedding is also a \hj-for-strongness embedding, as was shown 
in the proof of lemma \ref{lemma.VThetaElemMjkappa}.
\qed
\end{proof}

As a closing observation, note that analogues of many of the results in this 
section can 
be proven when $V$ is replaced by a more general model, $N$.

\section{Large cardinals strictly above a \Vopenka\ 
cardinal}\label{section.CardinalStrengths}

In the next few sections, I define the remaining large cardinals mentioned in 
the chart from the introduction, and I prove  
results 
about their consistency and implicational strengths. The sections are organized 
in order of strength in the large cardinal hierarchy. In the present section, I 
consider cardinals stronger than a \Vopenka\ cardinal but no stronger than an 
almost-huge cardinal.

I begin by defining the large cardinal notions that I will be analyzing in this 
section, starting with the \hj\ order and the \shj\ cardinals. These 
definitions 
are 
somewhat analogous to the definitions of the many times huge and superhuge 
cardinals, which are defined in 
\cite{BarbanelDipriscoTan84:nHugeSuperhuge}.

\begin{definition} \label{definition.shj}
Given an ordinal $\eta$,
the cardinal $\kappa$ has \textbf{\hjot{\eta}} if and
only if there exists a strictly increasing sequence 
$\<
\theta_\alpha \st \alpha < \eta >$ of ordinals such 
that for
each ordinal $\alpha < \eta$, there exists a \hj\ embedding for $\kappa$  with 
\supname\ 
$\theta_\alpha$.
 The cardinal $\kappa$ is \textbf{super high jump} if and only if there exist 
 \hj\ 
 embeddings for $\kappa$ of arbitrarily high clearance. (In other words, a 
\shj\ 
 cardinal $\kappa$ has
\hjot{\ORD}.)
\end{definition}

The \ahj\ order and the \sahj\ cardinals are defined similarly to the \hj\ 
order 
and the \shj\ cardinals, as follows.

\begin{definition}
Given an ordinal $\eta$, the cardinal $\kappa$ has \textbf{almost-high-jump 
order} 
$\eta$ if and
only if there exists  a 
strictly increasing 
sequence $\<
\theta_\alpha \st \alpha < \eta >$ of cardinals such 
that for
each ordinal $\alpha < \eta$, there exists an \ahj\ embedding for $\kappa$ with 
\supname\ $\theta_\alpha$. The cardinal 
$\kappa$ is \textbf{super almost high jump} 
if and only if there exist \ahj\ embeddings of arbitrarily high clearance for 
$\kappa$.
\end{definition}

It will also be interesting to consider \hj\ embeddings with \textbf{excess 
closure}, that is, embeddings $j: V \to M$ with 
clearance $\theta$ such that the target model $M$ is closed under sequences of 
length greater than $\theta$. For instance, \hj\ embeddings with clearance 
$\theta$ where the target model is closed under sequences of length $2^\theta$ 
will be fruitful objects of study. An extreme example of excess closure is as 
follows.

\begin{definition}\label{definition.StronglySHJ}
The cardinal $\kappa$ is \textbf{\hjp\ with unbounded excess closure} if and 
only if for 
some fixed clearance
$\theta$, for all cardinals $\lambda \geq \theta$, there is a \hj\ measure 
on 
$P_\kappa \lambda$ generating an embedding with clearance $\theta$.
\end{definition}

With all of the above definitions given, the time has come to prove many of the 
simpler 
consistency 
strength relations shown on the chart in the introduction, along with some 
additional related consistency strength relations that are not shown on the 
chart.

I begin with the following proposition, which involves a \hj\ embedding with a 
little bit of excess closure. This proposition is a simple example of 
the use of lemma \ref{lemma.VThetaElemMjkappa}, which will be used in 
more complicated arguments later.

\begin{proposition} \label{proposition.HJplus}
Suppose that there exists a pair of cardinals $(\kappa, \theta)$ such that 
there 
is a \hj\ embedding $j: V \to M$ with critical point $\kappa$ and clearance 
$\theta$ and such that $M^{2^\theta} \of M$. Then the cardinal $\kappa$ is 
\shjp\ 
in 
the model $V_\theta$, and the cardinal $\kappa$ has \hj\ order $\theta$ in $V$. 
Furthermore,  there are many \shj\ cardinals in the models $V_\kappa$, 
$V_\theta$, and $M_{j(\kappa)}$.
\end{proposition}

\begin{proof}
By lemma \ref{lemma.HJFactor}, there is a factor embedding, $h$, of $j$ such 
that $h$ has clearance $\theta$ and is generated by a \hj\ measure $U$ on 
$P_\kappa \theta$. By lemma 
\ref{lemma.ClearanceFactsAHJEmbedding}, the cardinal exponentiation identity 
$\theta^\kappa = \theta$ holds. It follows that the model $M$ is sufficiently 
closed so that $U \in M$.

In the model $M_{j(\kappa)}$, consider the set of cardinals $\lambda$ such that 
there is a \hj\ measure generating an embedding with 
critical point $\kappa$ and clearance $\lambda$. By lemma 
\ref{lemma.VThetaElemMjkappa}, the elementarity relation $V_\theta \elem 
M_{j(\kappa)}$ holds. It follows that if this set of cardinals is bounded
in the model $M_{j(\kappa)}$, then this bound is below $\theta$. But $\theta$ 
is 
an element of this set, since $U \in M$. 
Therefore, the set is unbounded in both $V_\theta$ and $M_{j(\kappa)}$, and in 
particular, $\kappa$ is a \shj\ cardinal in the model $M_{j(\kappa)}$. By 
reflection, there are many \shj\ cardinals in the model $V_\kappa$. By the 
elementarity of $j$ and since $V_\theta \elem M_{j(\kappa)}$, it follows that 
there are also many \shj\ cardinals in $M_{j(\kappa)}$ and in $V_\theta$. 
Finally, since $V_\theta \satisfies \ZFC$ and since every \hj\ measure of 
$V_\theta$ is 
also a \hj\ measure in $V$, it follows that the cardinal $\kappa$ has \hj\ 
order 
$\theta$ in $V$. 
\qed 
\end{proof}

In later similar consistency proofs, I will finish the proof with a conclusion 
about one of $M_{j(\kappa)}$, $V_\kappa$, or $V_\theta$, and 
leave it to the reader to work out the additional consequences in the other 
models.
Note that the hypothesis of proposition \ref{proposition.HJplus} is equivalent 
to the hypothesis that there for some pair $(\kappa, \theta)$, such that there 
is a \hj\ measure on $P_\kappa 2^\theta$. This alternative hypothesis follows 
immediately from the hypothesis of proposition \ref{proposition.HJplus}. For 
the 
converse, given a pair $(\kappa, \theta)$ such that there is a \hj\ measure 
on $P_\kappa 2^\theta$, the clearance of the corresponding embedding must be at 
most $\theta$. If the clearance of this embedding is some $\theta'< 
\theta$, then take a $2^{\theta'}$-supercompactness factor embedding and apply 
lemma \ref{lemma.HJFactor}.

Next, I will consider elementary embeddings for which the closure of the target 
model is extremely 
large compared with the clearance of the embedding, beginning with the
\hj\ cardinals with \uec.

\begin{proposition} \label{proposition.AlmostHuge>SSHJ}
Suppose the cardinal $\kappa$ is almost huge. Then in the model $V_\kappa$, 
there are many cardinals $\delta$ such that $\delta$ is \hjuecp\
\end{proposition}

\begin{proof}
Suppose $\kappa$ is almost huge, witnessed by an 
elementary 
embedding $j: V \to M$ with clearance 
$\theta$. In particular, the embedding $j$ is also a \hj\ embedding.
Let $\lambda$ be a cardinal such that  $\theta \leq \lambda < j(\kappa)$. The 
cardinal $j(\kappa)$ is a strong limit cardinal. Therefore, by 
lemma 
\ref{lemma.HJFactor}, the embedding $j$ has a $\lambda$-supercompactness factor 
embedding with clearance $\theta$ generated by a \hj\ measure on $P_\kappa 
\lambda$. This \hj\ measure is an element of $M_{j(\kappa)}$. 
\qed
\end{proof}

Consider a cardinal $\kappa$ such that for all sufficiently large cardinals 
$\lambda$, there is a \hj\ measure on $P_\kappa \lambda$. It may be possible 
that such a cardinal is not \hjuecp, because the \hj\ measures may not all 
generate embeddings with the same closure. However, the following proposition 
shows that these two types of cardinals are equiconsistent. 

\begin{proposition} \label{proposition.SSHJEquiconsistentWithHJplusORD}
The following two large cardinal axioms are equiconsistent over \ZFC.

\begin{enumerate}
\item \label{item.HJplusORD}
There exists a cardinal $\kappa$ such that for all sufficiently large cardinals 
$\lambda$, there is a \hj\ measure on $P_\kappa \lambda$.
\item\label{item.SSHJ}
There exists a cardinal that is \hjuecp\
\end{enumerate}

In particular if there are \hj\ measures on $P_\kappa \lambda$ for all 
sufficiently large cardinals $\lambda$, then either $\kappa$ is \hjuecp\ or 
else 
there is a cardinal $\theta$ such that $\kappa$ is \hjuecp\ in the model 
$V_\theta$.
\end{proposition}
\begin{proof}

It is immediate from the definitions that if $\kappa$ is \hjuecp, then for all 
sufficiently large $\lambda$, there is a \hj\ measure on $P_\kappa \lambda$. 

For the converse, 
suppose that for all sufficiently large $\lambda$, there is a \hj\ measure on 
$P_\kappa \lambda$, but the cardinal $\kappa$ is not \hjuecp.
Let $\theta_0$ be the minimal cardinal such that for all cardinals $\lambda 
\geq 
\theta_0$, 
there is a \hj\ measure on $P_\kappa \lambda$. Since the cardinal $\kappa$ is 
not \hjuecp, these \hj\ measures do not all 
generate embeddings with clearance $\theta_0$. 

None of these measures generates a \hj\ embedding with clearance less than 
$\theta_0$. If it did, then the minimality of $\theta_0$ would be contradicted 
by taking factor embeddings and applying lemma \ref{lemma.HJFactor}.

Accordingly, let $\theta_1$ be the least 
cardinal above 
$\theta_0$ such that there is a \hj\ embedding for $\kappa$ with clearance 
$\theta_1$.  Let $j: V \to M$ be a \hj\ embedding for $\kappa$ with clearance 
$\theta_1$. Then the model $V_{\theta_1}$ satisfies \ZFC\ by lemma 
\ref{lemma.VThetaElemMjkappa}, and in this model, the cardinal $\kappa$ 
is \hjuecp\ with respect to the clearance 
$\theta_0$.
\qed
\end{proof}

Next, proposition \ref{proposition.ClosureHierarchy} shows that the degrees of 
excess 
closure of \hj\ embeddings form a hierarchy of consistency strength. In this 
hierarchy, there are many more cardinals above the ones described in proposition 
\ref{proposition.ClosureHierarchy} and below the \hj\ cardinals with unbounded 
excess closure. For further details, see \cite[pp. 
117-118]{Perlmutter2013:Dissertation}.  

\begin{proposition} \label{proposition.ClosureHierarchy}
Suppose that for some cardinals $\kappa$ and $\theta$ and for some ordinal 
$\alpha< \theta$, there exists a \hj\ 
embedding $j: V \to M$ with critical point $\kappa$ and clearance $\theta$ such 
that the model $M$ is closed under sequences of length 
$2^{\aleph_{\theta+\alpha}^{<\kappa}}.$ Then in the model $M_{j(\kappa)}$, 
there 
are unboundedly many cardinals $\lambda$ such that there is 
a \hj\ measure on $P_\kappa (\aleph_{\lambda + \alpha})$ generating a \hj\ 
embedding with critical point $\kappa$ and clearance $\lambda$.
\end{proposition}

\begin{proof}
In the model $M_{j(\kappa)}$, consider the set of cardinals $\lambda$ such that 
there is 
a \hj\ measure on $P_\kappa (\aleph_{\lambda + \alpha})$ generating a \hj\ 
embedding with critical point $\kappa$ and clearance $\lambda$. The model 
$M_{j(\kappa)}$ is sufficiently closed to see that $\theta$ is an element of 
this 
set. By lemma \ref{lemma.VThetaElemMjkappa}, the elementarity relation 
$V_\theta \elem M_{j(\kappa)}$ holds, so it follows that this set is unbounded 
in $M_{j(\kappa)}$.
\qed
\end{proof}

Next, I move on to prove some results lower down in the hierarchy of \hj\ 
cardinals 
and related cardinals.

\begin{proposition}\label{proposition.HJorder} \label{proposition.HJplusOrder}
Let $\eta$ and $\eta'$ be ordinals such that $\eta < \eta'$. Suppose 
the cardinal $\kappa$ has \hj\ order $\eta'$. Then there is an 
elementary embedding $j: V \to M$ with critical point $\kappa$ such that the 
cardinal $\kappa$ has \hj\ order $\eta$ in $M_{j(\kappa)}$. 
\end{proposition}

\begin{proof}
The cardinal $\kappa$ has \hj\ order $\eta'$, and this is witnessed 
by a sequence of clearances $\langle \theta_\alpha \st {\alpha<\eta'\rangle}$.
Let $j: V \to M$ be a \hj\ embedding for $\kappa$ with clearance $\theta$ for 
some $\theta_\alpha$ such that $\alpha \geq \eta$. Then the model $M$ is 
sufficiently closed so that in $M_{j(\kappa)}$, the cardinal $\kappa$ has 
\hj\ order $\eta$.
\qed
\end{proof}


I now move further down the large cardinal hierarchy, to the \ahj\ cardinals. 
Recall from the introduction that \ahj\ cardinals are characterized by 
combinatorially by coherent sequences of normal measures, which are described in 
detail in \cite[lemma 57]{Perlmutter2013:Dissertation}.

\begin{proposition}\label{proposition.HJ>SAHJ}
Suppose there is a \hj\ embedding with critical point $\kappa$ and clearance 
$\theta$. Then $\kappa$ has \ahj\ order $\theta$, and in the models $V_\theta$, 
$M_{j(\kappa)}$, and $V_\kappa$, there are 
many \sahj\ cardinals.
\end{proposition}

\begin{proof}
Suppose  $j: V \to M$ is a \hj\ embedding with critical point $\kappa$ and 
clearance $\theta$.  It follows immediately from definitions that the embedding 
$j$ also witnesses that $\kappa$ is almost high jump.  By
corollary 
\ref{lemma.ThetaStrongLimit}, the cardinal 
$\theta$ is a 
strong limit,
and so it follows that the coherent sequence of measures witnessing that there 
is an \ahj\ embedding for $\kappa$ with clearance 
$\theta$  is an 
element of
$H_{\theta^+}$. This coherent sequence of measures is also an element of $M$, 
by the
closure of $M$. Therefore, the cardinal $\kappa$ is almost high jump 
in $M$ with clearance $\theta$. 
Consider the set $\set{\delta \st M_{j(\kappa)} \satisfies 
\kappa \text{  
almost high jump with clearance }\delta}.$  By theorem 
\ref{Theorem.M_thetaElemInM_j(kappa)}, if this set 
has a bound in 
$M_{j(\kappa)}$, then the bound must be less than 
$\theta$. It 
follows 
that the set is unbounded in $M_{j(\kappa)}$, and so $\kappa$ is super almost 
high jump 
in the model
$M_{j(\kappa)}$. The other conclusions are immediate.
\qed
\end{proof}

\begin{proposition}\label{corollary.AHJorder}
Let $\eta < \eta'$ be ordinals. Suppose 
the cardinal $\kappa$ has \ahj\ order $\eta'$. Then there is an 
elementary embedding $j: V \to M$ with critical point $\kappa$ such that the 
cardinal $\kappa$ is has \ahj\ order $\eta$ in $M_{j(\kappa)}$. 
\end{proposition}
\begin{proof}
The proof follows the same reasoning as the proof of proposition 
\ref{proposition.HJorder}, replacing \hj\ embeddings with \ahj\ embeddings and 
\hj\ measures with coherent sequences of measures.
\qed
\end{proof}

Finally, I reach the Shelah-for-supercompactness cardinals.

\begin{proposition} \label{proposition.AHJ>SSC}
Suppose the cardinal $\kappa$ is \ahj. Then there are many cardinals below 
$\kappa$ that are Shelah for supercompactness.
\end{proposition}
\begin{proof}
Suppose $j: V \to M$ is an \ahj\ embedding for $\kappa$ with clearance 
$\theta$. I will show 
that $\kappa$ is 
Shelah for
supercompactness in $M$. Let $f: \kappa \to \kappa$ 
be a
function in $M$.
Let $j_0: V \to M_0$ be the 
$\lambda$-supercompactness factor
embedding induced by $j$ via the seed $j \image \lambda$, where $\lambda$ is the 
maximum of 
$j(f)(\kappa)$ and $\kappa$.
From corollary 
\ref{lemma.ThetaStrongLimit}, it follows 
that 
$2^{\lambda^{<\kappa}} < \theta$. Therefore, the 
$\lambda$-supercompactness measure $U$ that 
generates $j_0$
is an element of $M$. By reasoning similar to the proof of lemma 
\ref{lemma.FactorEmbeddingEquality}, it 
follows 
that $j(f)(\kappa) = 
j_0(f)(\kappa)$. 
Let $j^M_0$ be the elementary embedding generated by 
$U$ in $M$. The 
measure $U$ is an element of $V_\theta = M_\theta$, 
which  satisfies 
$\ZFC$ by theorem 
\ref{Theorem.M_thetaElemInM_j(kappa)}.  It follows 
that $j_0 \restrict V_\theta = j_0^M \restrict 
M_\theta$. 
Therefore, in $M$, the elementary embedding $j_0^M$
witnesses that $\kappa$ is Shelah for 
supercompactness with
respect to the function $f$. Since $f$ was 
arbitrary, it
follows that $\kappa$ is Shelah for supercompactness 
in $M$.
\qed
\end{proof}
I wind up the section with a few miscellaneous propositions. 
Proposition \ref{proposition.LeastHJnotStr} shows that several direct 
implications are lacking from the large cardinal hierarchy.

\begin{proposition}\label{proposition.LeastHJnotStr}
The least \hj\ cardinal is not $\Sigma_2$-reflecting. In particular, it is not 
supercompact and not even strong. The same is true for the least almost-huge 
cardinal, the least \ahj\ cardinal, and the least Shelah-for-supercompactness 
cardinal.
\end{proposition}

\begin{proof} 
All of these cardinals can be characterized by $\Sigma_2$ definitions --- they 
are characterized by a measure or a set of measures with certain 
combinatorial 
properties, all of which can be seen from within a particular $V_\alpha$. Since 
supercompact and strong cardinals are 
$\Sigma_2$-reflecting, the theorem follows.
\qed
\end{proof}

The proofs of propositions \ref{lemma.AHJStrongerThanSC} and 
\ref{proposition.HJO2} are simple and are left as exercises for the reader. The 
complete proofs can be found in \cite[propositions 81 and 
83]{Perlmutter2013:Dissertation}
\begin{proposition} \label{lemma.AHJStrongerThanSC}
\label{proposition.AHJStrongerThanSC}
Suppose $j: V \to M$ is an elementary embedding with clearance $\theta$ 
witnessing that $\kappa$ is \ahj. Then in the model $V_\kappa$ there are many 
supercompact cardinals.
\end{proposition}


The definition of super high jump makes it tempting to think 
that every cardinal that is both supercompact and high jump is super high jump. 
However, the 
following simple 
proposition shows that this is not the case.

\begin{proposition} \label{proposition.HJO2}
If $\kappa$ is the least cardinal that has \hjot{2}, 
then in 
$V_\kappa$, there are many cardinals that are both supercompact and high jump 
but not 
super high jump.
\end{proposition}

%

Proposition \ref{proposition.HJO2} shows that below a cardinal of \hj\ order 
$2$, there are 
many cardinals that are both \hjp\ and tall. If a cardinal $\kappa$ is both 
\hjp\ 
and tall, then there are \hj\ embeddings for $\kappa$ such that $j(\kappa)$ is 
arbitrarily large --- this can be seen by taking a \hj\ embedding for $\kappa$ 
followed by a tallness embedding for $j(\kappa)$. It is easy to see that below 
a 
cardinal that is both \hj\ and supercompact, there are many \hj\ cardinals. But 
it is an open question whether the existence of a cardinal that is both \hj\ 
and 
tall is equiconsistent with the existence of a \hj\ cardinal.

Many more definitions could be made along the lines of the ones given in this 
section, and these definitions would lead to many more questions of consistency 
strength. In light of proposition \ref{proposition.HJplus}, the following 
question comes to mind.

\begin{question} \label{question.ClosureTheta+} What is the consistency 
strength 
of the existence of a \hj\ embedding
with critical point $\kappa$ and clearance $\theta$ generated by a \hj\ measure 
on $P_\kappa \theta^+$? Is the 
existence of such an embedding equiconsistent with the existence of a \hj\ 
embedding with critical point $\kappa$ and clearance $\theta$ generated by a 
\hj\ 
measure on $P_\kappa 2^{\theta}$?
\end{question}
Of course, under \GCH, these two types of embeddings are equivalent.

\section[Equivalence of \Vopenka\ and Woodin-for-s.c. cardinals]{The 
equivalence of \Vopenka\ cardinals and Woodin-for-supercompactness
cardinals}\label{section.WoodinSC=Vopenka}

In this section, I define a Woodin-for-supercompactness cardinal and show that 
it is strictly weaker than a Shelah-for-supercompactness cardinal. I then show 
that a cardinal is Woodin for supercompactness if and only if it is a \Vopenka\ 
cardinal.

The definition of a Woodin-for-supercompactness cardinal is as follows. This 
definition 
is taken from \cite[2]
{ApterSargsyan2007:ReductionInCons(WoodinSC)}. These cardinals have also been 
studied by Foreman 
\cite[p.31]{Foreman07:GenericElemEmb3} and by Fuchs 
\cite[p.1043]{Fuchs09:CombinedMaximality} under the name of Woodinized 
supercompact cardinals. 
A Woodin-for-supercompactness cardinal is the analogue of a Woodin cardinal, 
with supercompactness in place of strongness.

\begin{definition} \label{definition.WSC} A cardinal 
$\delta$ is \textbf{Woodin for
supercompactness} if and only if for every function 
$f: \delta
\to \delta$, there exists a cardinal $\kappa < \delta$ such 
that $\kappa$ is a 
closure point of $f$ (i.e.\ $f \image
\kappa \of \kappa$), and there exists an elementary 
embedding $j:
V \to M$ with critical point $\kappa$ 
such that
$M^{j(f)(\kappa)} \of M$. 
\end{definition}

Note in particular that every Woodin-for-supercompactness cardinal is also a 
Woodin cardinal.

Apter and Sargsyan require that $f$ be defined so 
that $f(\alpha)$ 
is always a cardinal and so that the elementary 
embedding $j$ is 
generated by a 
supercompactness measure on $P_\kappa \lambda$ for 
some $\lambda < 
\delta$, but 
this definition is equivalent to the definition given here. 
The first requirement does not make their definition any weaker, because  any 
function $f$ can be replaced with a function $f'$ 
given by 
$f'(\alpha) = |f(\alpha)|^+$. Theorem
\ref{theorem.WSCcharacterization} below shows that the second requirement does 
not make 
their 
definition any stronger.

The next theorem shows that a Shelah-for-supercompactness 
cardinal is strictly stronger than a Woodin-for-supercompactness 
cardinal. The 
proof is an improvement on lemma 
1.1 of
\cite{ApterSargsyan2007:ReductionInCons(WoodinSC)} 
and 
uses a similar line of reasoning. 

\begin{theorem}
\label{proposition.ShelahSCStrongerThanWoodinSC}
\label{theorem.ShelahSCStrongerThanWoodinSC}
Suppose the cardinal $\kappa$ is Shelah for supercompactness. 
Then $\kappa$ is 
Woodin for supercompactness, and there are many 
cardinals below 
$\kappa$ that are Woodin for supercompactness in both $V_\kappa$ and $V$.
\end{theorem}

\begin{proof}

Let $\kappa$ be Shelah for supercompactness. The main difficulty is to show 
that 
$\kappa$ is Woodin for supercompactness. Once this has been shown, it is 
immediate that there are many cardinals below $\kappa$ that are Woodin for 
supercompactness in the model $V_\kappa$, because ``The cardinal $\kappa$ is 
Woodin for 
supercompactness'' is $\Pi_1^1$-definable over $V_\kappa$, and 
Shelah-for-supercompactness cardinals are $\Pi_1^1$-indescribable, since they 
are weakly 
compact. Furthermore, it is easily seen that any cardinal that is Woodin for 
supercompactness in the model $V_\kappa$ must also be Woodin for 
supercompactness in $V$.

Towards showing that $\kappa$ is Woodin for supercompactness, let $f: \kappa 
\to 
\kappa$ be an 
arbitrary 
function. I will show that $\kappa$ satisfies the 
definition of 
Woodin for supercompactness with respect to $f$. Without loss of generality, I 
assume that $f$ is nowhere regressive, that is to say, for all $\alpha < 
\kappa$ the inequality $\alpha \leq f(\alpha)$ holds. It follows that $\kappa 
\leq j(f)(\kappa).$ 
Let $g: \kappa \to \kappa$ be given by $g(\alpha) = 
\scexp{f(\alpha)}{\alpha}$. Let $j:V \to M$ witness that $\kappa$ is 
Shelah for 
supercompactness with respect to the function $g$, that is, the embedding $j$ 
has critical 
point $\kappa$ and $M^{j(g)(\kappa)} \of M$. Note that $j(g)(\kappa) = \scexp 
{(j(f)(\kappa))}{\kappa}$, as calculated in both $M$ and $V$.

Let $U$ be the normal 
fine measure 
on $P_\kappa \big(j(f)(\kappa)\big)$ induced via $j$ by the seed $j(f)(\kappa)$. 
Let 
$j_U: V \to N$ be the $j(f)(\kappa)$-supercompactness 
factor 
embedding induced by $U$. Let 
$k: N 
\to M$ be 
the elementary embedding such that $k \circ h = j$. 
The model $M$ is 
sufficiently closed so that $U \in M$, and the closure of $M$ further 
guarantees 
that every function from $P_\kappa\big(j(f)(\kappa)\big)$ to $M$ is an element 
of $M$. It 
follows that the elementary embedding induced by $U$ in $M$, as calculated in 
$M$, is equal to $j_U \restrict M$, as calculated in $V$.

\begin{diagram}[w=5em]
 V & \rTo^j
& M & \rTo^{j_U \restrict M} & N'  \\
\dTo<{j_U} & \ruTo>{k} & \\
 N & &\\ 
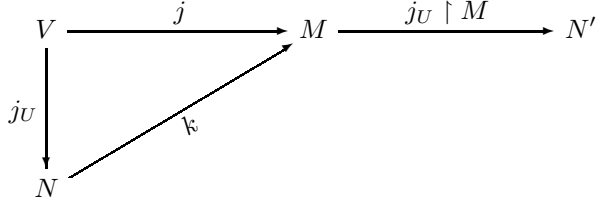
& & \figlist{Factor embeddings of a Shelah-for-supercompactness embedding}
\end{diagram}


Next, I will show that 
\begin{align} \label{thingie}
j(f)(\kappa) = j_U(j(f))(\kappa)
\end{align} 

First of all, $f = j(f) \restrict \kappa$. Applying $j_U$ to both sides, it 
follows that $j_U(f) = j_U(j(f)) \restrict j_U(\kappa),$ and 
in particular, $j_U(f)(\kappa) = j_U(j(f))(\kappa).$ 
Furthermore, by reasoning similar to the proof of lemma 
\ref{lemma.FactorEmbeddingEquality}, it follows that 
$j(f)(\kappa) = j_U(f)(\kappa)$, and so statement \ref{thingie} is proven.

The embedding $j_U \restrict M$ is generated in $M$ by the measure $U \in M$. 
Therefore, in 
$M$, the cardinal $j(\kappa)$ satisfies the 
definition of Woodin for 
supercompactness with respect to the function $j(f)$ 
--- this fact is 
witnessed by the embedding $j_U \restrict M$ along with statement \ref{thingie},
since 
$\kappa$ is a 
closure point of 
$j(f)$. It follows from the elementarity of $j$ that in $V$, the cardinal 
$\kappa$ satisfies the 
definition of 
Woodin for supercompactness with respect to the 
function $f$. But 
$f$ was chosen arbitrarily, so $\kappa$ is Woodin 
for 
supercompactness in $V$.
\qed
\end{proof}

Like Woodin cardinals, Woodin-for-supercompactness cardinals have several 
alternative characterizations, as described below in theorem 
\ref{theorem.WSCcharacterization}. The proof that these characterizations hold 
is essentially the same as for the case of Woodin cardinals, (see \cite[Theorem 
26.14]{Kanamori:TheHigherInfinite2ed}). In order to state them, I 
need the following definition of $(\gamma, A)$\nobreakdash-supercompactness.

\begin{definition}
Given a set $A$ and a cardinal $\gamma$, the cardinal $\kappa$ is 
$\mathbf{( \boldsymbol{\gamma}, A)}$\textbf{-supercompact}  if and only if 
there 
is an 
elementary 
embedding $j: V \to M$ 
with critical point $\kappa$ such that

\begin{enumerate}
\item $\gamma < j(\kappa)$, 
\item $M^\gamma \of M$, and 
\item $A \cap V_{\gamma} = j(A) \cap 
V_{\gamma}.$
\end{enumerate}
\end{definition}

Given cardinals $\kappa$ and $\delta$ such that $\kappa < \delta$, the notation 
\emph{$\kappa$ is $(<\!\!\delta, 
A)$-supercompact} denotes that $\kappa$ is $(\gamma, A)$-supercompact for all 
cardinals
$\gamma < \delta$.
The alternative characterizations of Woodin-for-supercompactness cardinals are 
as follows.

\begin{theorem}\label{theorem.WSCcharacterization}
Given a cardinal $\delta$, the following are 
equivalent.

\begin{enumerate}
\item\label{hypothesis.WSCstrict}

For every function 
$f: \delta
\to \delta$, there exists a cardinal $\kappa < \delta$ such 
that $\kappa$ is a 
closure point of $f$ (i.e.\ $f \image
\kappa \of \kappa$), and there exists an elementary 
embedding $h:
V \to M$ with critical point $\kappa$ 
such that
$M^{h(f)(\kappa)} \of M$. Furthermore, the embedding $h$ is generated by a 
normal fine measure on $P_\kappa \lambda$ for some cardinal $\lambda < \delta$.

\item \label{hypothesis.WSC} The cardinal $\delta$ 
is 
Woodin for 
supercompactness. That is to say, for every function 
$f: \delta
\to \delta$, there exists a cardinal $\kappa < \delta$ such 
that $\kappa$ is a 
closure point of $f$ (i.e.\ $f \image
\kappa \of \kappa$), and there exists an elementary 
embedding $j:
V \to M$ with critical point $\kappa$ 
such that
$M^{j(f)(\kappa)} \of M$. 
\item \label{hypothesis.ManyASC} For every set $A \of 
V_\delta$, the set 
$$\set{\kappa< \delta \st \kappa \text{ is } 
(<\!\!\delta, A) \text{-supercompact}}$$ is stationary in $\delta$. 
\item \label{hypothesis.OneASC} For every set of ordinals $A \of 
\delta$, the set 
$$\set{\kappa< \delta \st \kappa \text{ is } 
(<\!\!\delta, A) \text{-supercompact}}$$ is nonempty.

\end{enumerate}
\end{theorem}

\begin{proof}
This is the analogue of a standard theorem for Woodin cardinals, and the proof 
is essentially the same as for that theorem. See \cite[theorem 
88]{Perlmutter2013:Dissertation} for details.
\qed
\end{proof}

I now shift my attention to the \Vopenka\ cardinals, which I will eventually 
show are equivalent to the Woodin-for-supercompactness cardinals.

\begin{definition}\label{definition.Vopenka}
The cardinal $\delta$ is \textbf{\Vopenka} if and only if for every 
$\delta$-sequence 
of 
model-theoretic structures $\langle M_\alpha \st \alpha< \delta 
\rangle$ over the same language, with each structure $M_\alpha$ an element of 
$V_\delta$, there exists an elementary embedding $j: M_\alpha \to M_\beta$ for 
some ordinals $\alpha < \beta < \delta$. 
\end{definition}

The following result is well-established in the large cardinal 
literature.

\begin{proposition}
\label{proposition.LeastVopenkaNotWC} 
\label{proposition.LeastWoodinForSCNotMeas}
The least cardinal that is \Vopenka\ is not weakly 
compact. 
\end{proposition}

\begin{proof}
``The cardinal $\kappa$ is \Vopenka'' is definable by 
a 
$\Pi_1^1$ formula over $V_\kappa$, but weakly compact 
cardinals are $\Pi_1^1$-indescribable.
\qed
\end{proof}

It is convenient to have a characterization of \Vopenka\ cardinals in terms of 
a more limited class of model-theoretic structures. Towards this end, following 
\cite{Kanamori:TheHigherInfinite2ed}, I make the following definition.

\begin{definition}
Let $\delta$ be a cardinal. A sequence of model-theoretic structures, $\< 
M_\alpha \st \alpha < \delta>$ is a \textbf{natural $\delta$-sequence} if 
and only if the 
following properties are satisfied. There is a function $f: \delta \to \delta$ 
such that the domain of $M_\alpha$ is $V_{f(\alpha)}$ and such that whenever 
$\alpha < \beta < \delta$, are ordinals, it follows that $\alpha < f(\alpha) 
\leq f(\beta) < \delta$. Furthermore, each structure $M_\alpha$ is of the form 
$(V_{f(\alpha)}, \in, \singleton{\alpha}, R_\alpha)$, for some unary relation 
$R_\alpha \of V_{f(\alpha)}$. 
\end{definition}

Without loss of generality, $R_\alpha$ may be taken to 
encode finitely many constants,  
functions, and relations of any finite arity on $V_{f(\alpha)}.$ I will show in 
proposition \ref{proposition.NaturalSequence} below that it suffices to 
consider 
only natural sequences when 
determining whether the cardinal $\kappa$ is \Vopenka. 

It is immediately clear that any sequence of the type specified in definition 
\ref{definition.Vopenka}  can be encoded as a natural sequence if the language 
involved is countable. For larger languages, one might be concerned that 
the critical point of the embedding would be small enough to mess up the 
encoding. 

To clear up this concern, I will discuss the critical point of an 
elementary 
embedding between two elements of a natural sequence. This analysis will also 
play an important role in the proof of the equivalence between \Vopenka\ 
cardinals and Woodin-for-supercompactness cardinals. The inclusion of the 
constant $\singleton{\alpha}$ in the definition of a natural sequence ensures 
that any elementary embedding between 
members of a natural sequence has a critical point. 

I begin the analysis of 
these critical points with the definition of the \Vopenka\ filter, given below 
in 
definition \ref{definition.VopenkaFilter}.
The \Vopenka\ filter is actually a filter on $\delta$ if and only if $\delta$ 
is 
a 
\Vopenka\ cardinal. In this case, the \Vopenka\ filter is a normal filter and 
contains every club.\footnote{Every normal \emph{ultrafilter} on a cardinal 
$\delta$ contains every club. But a normal \emph{filter} on a cardinal $\delta$ 
contains every club if and only if it contains every tail. The \Vopenka\ filter 
on a \Vopenka\ cardinal is not in general an ultrafilter, by proposition 
\ref{proposition.LeastVopenkaNotWC}.} These facts are 
proven in 
\cite[pp.336-337]{Kanamori:TheHigherInfinite2ed} and \cite[proposition 
6.3]{SolovayReinhardtKanamori1978:StrongAxiomsOfInfinity}.

\begin{definition}[{\cite[p.336]{Kanamori:TheHigherInfinite2ed}}]
\label{definition.VopenkaFilter}
Let $\delta$ be an inaccessible cardinal. Then the set $X \of \delta$ is a 
member 
of the \textbf{\Vopenka\ filter} on $\delta$ if and only if there exists a 
 natural $\delta$-sequence $\<M_\alpha \st \alpha < \delta>$ such 
that 
whenever $j: M_\alpha \to M_\beta$ is an elementary embedding, then the 
critical 
point of $j$ is an element of $X$. 
\end{definition}

%
%
%
Given an elementary embedding $j: V_\alpha \to V_\beta$, the critical point of 
$j$ must be inaccessible, and so if  $\delta$ is a \Vopenka\ cardinal, then the 
set of inaccessible cardinals below $\delta$ is a member of the \Vopenka\ 
filter 
on $\delta$. In particular, this implies that $\delta$ is Mahlo, since the 
\Vopenka\ filter contains every club.

The next proposition states that it suffices to consider only natural sequences 
in defining \Vopenka\ cardinals.

\begin{proposition} \label{proposition.NaturalSequence}
The cardinal $\delta$ is \Vopenka\ if and only if for every normal 
$\delta$-sequence $\< M_\alpha \st \alpha < \delta>$, there is an elementary 
embedding $j: M_\alpha \to M_\beta$ for some ordinals  
$\alpha, \beta < \delta$. 
\end{proposition}

\begin{proof}
The forwards direction is immediate. For the converse, 
 the essential observation is that the \Vopenka\ filter on $\delta$ contains 
 every tail, so that after encoding a sequence of model-theoretic structures as 
 a natural sequence, it is possible to choose an embedding with critical point 
 much larger than the size of the language of the original structures, so that 
 the embedding does not interfere with the encoding. 
 \qed
\end{proof}

Kanamori suggests the equivalence of a 
Woodin-for-supercompactness cardinal to a \Vopenka\ cardinal in
\cite[p.364]{Kanamori:TheHigherInfinite2ed}. However, 
he does not formally define a Woodin-for-supercompactness cardinal, nor does he 
work out the details of the equivalence.

\begin{theorem} \label{theorem.WSC=Vopenka}
The cardinal $\delta$ is Woodin for supercompactness if and only if $\delta$ is 
a \Vopenka\ cardinal. Furthermore, if $\delta$ is a \Vopenka\ cardinal, then 
for 
every set 
$A 
\of 
V_\delta$, the set $\set{\kappa < \delta \st \kappa \text{ is $(<\!\!\delta, 
A)$-supercompact}}$ is a member of the \Vopenka\ filter on $\delta$.
\end{theorem}

\begin{proof}
For the forward direction, let $\delta$ be Woodin for
supercompactness, and let $A = \langle A_\alpha \st \alpha < \delta \rangle$ be 
a 
natural $\delta$-sequence. I will show that for some ordinals
$\alpha <
\beta < \delta$, there exists an elementary embedding 
$j:
A_\alpha \to A_\beta.$ 

Since $\delta$ is Woodin for 
supercompactness, there is a cardinal $\kappa < \delta$ such that $\kappa$ is 
$(<\!\!\delta, 
A)$-supercompact. Therefore, by choosing a large enough degree of 
$A$-supercompactness, there is an elementary embedding $j: V \to N$ such that 
$j(A)_\kappa = A_\kappa$ and $j \restrict A_\kappa  \in N$. In $N$, the map $j 
\restrict A_\kappa$ is an elementary embedding from $j(A)_\kappa$ to 
$j(A)_{j(\kappa)}$. So in $N$, there exists an elementary embedding between two 
elements of $j(A)$. By the elementarity of $j$, there exists an elementary 
embedding between two elements of $A$ in $V$. It follows that $\delta$ is 
\Vopenka.

The proof of the converse direction uses some of the same ideas as the proof of 
proposition 24.14 of 
\cite{Kanamori:TheHigherInfinite2ed}, which shows that 
$V_\delta$ contains many extendible cardinals if $\delta$ is \Vopenka. Suppose 
that the 
cardinal $\delta$ is \Vopenka, and let $A \of V_\delta$. I will 
show that there exists a cardinal $\kappa < \delta$ such that $\kappa$ is 
$(<\!\!\delta, 
A)$-supercompact, thereby 
showing that the cardinal $\delta$ is Woodin for supercompactness. Indeed, I 
will show that the set 
of such $\kappa$ is an element of the \Vopenka\ filter on $\delta$.

Let $g:\delta \to \delta$ be the variation of the 
failure-of-$A$-supercompactness function described as follows. Given $\xi< 
\delta$, let 
$g(\xi)$ be the least 
cardinal $\eta > \xi$ such that $\xi$ is not $(\eta, A)$-supercompact. 
In case no such $\eta$ exists, then set $g(\xi) = \xi$.

Let $C \of \delta$ be the club of closure points of $g$, i.e. $C = 
\set{\rho < \delta \st g \image \rho \of \rho}$. Since the \Vopenka\ filter on 
$\delta$ contains every club, it follows that the club $C$ is a member of this 
filer. 
Therefore, there exists a natural 
$\delta$-sequence  $\<M_\alpha \st \alpha < \delta>$ such that 
whenever $j: M_\alpha 
\to M_\beta$ is an elementary embedding, the critical point of $j$ is an 
element 
of $C$.

For each ordinal $\alpha < \delta$, let $\gamma_\alpha$ be the least 
inaccessible element 
of $C$ above 
all the ordinals of $M_\alpha$, and for each ordinal $\alpha < \delta$, let 
$$N_\alpha = ( V_{\gamma_\alpha}, \in, \singleton{\alpha}, M_\alpha, C \cap 
\gamma_\alpha, A \cap 
V_{\gamma_\alpha}).$$ Let $j: 
N_\alpha \to N_\beta$ be an elementary embedding. It suffices to show that the 
critical point, $\kappa$, of $j$ is $(<\!\!\delta, A)$-supercompact.

Assume to the contrary that $\kappa$ is not $(<\!\!\delta, A)$-supercompact. 
Then $\kappa< g(\kappa)$. Furthermore, 
$g(\kappa) < \gamma_\alpha$, because $\gamma_\alpha \in C$.
Since $M_\alpha$ is encoded in $N_\alpha$, it follows from the definition of 
the 
sequence $\<M_\alpha>$ that $\kappa \in C$. By the elementarity of $j$, it 
follows that $j(\kappa) \in C$ as well. 

Let $U$ be the normal fine measure on $P_\kappa \big(g(\kappa)\big)$ induced
via $j$ 
by the seed $j \image g(\kappa)$, and let $j_U: V \to N$ be the ultrapower 
generated by $U$.
Using the fact that $\gamma_\alpha$ is inaccessible, the theory of factor 
embeddings shows that there 
exists an elementary embedding $k$ so 
that the following diagram commutes. 
\begin{diagram}[w=5em]
V_{\gamma_\alpha} & \rTo^j & V_{j(\gamma_\alpha)} \\
\dTo^{j_U \restrict V_{\gamma_\alpha}} & \ruTo_k & \\
N \cap V_{\gamma_\alpha} & &\\
\end{diagram}
\figlist{Factor embeddings of a \Vopenka\ embedding}

I claim that 
the map $j_U : V \to N$ witnesses that $\kappa$ is $(
g(\kappa), A)$-supercompact. This will contradict the definition of $g$, 
thereby 
completing the proof. Clearly, this map is a $g(\kappa)$-supercompactness 
embedding with critical point $\kappa$, and so it suffices to show that $j_U(A) 
\cap V_{g(\kappa)} = A \cap V_{g(\kappa)}.$

First of all, since $A \cap V_{\gamma_\alpha}$ is encoded in $N_\alpha$, it 
follows 
that $j(A) \cap V_{g(\kappa)} = A \cap V_{g(\kappa)}$. Since the critical point 
of $k$ is an inaccessible cardinal above $g(\kappa)$, it follows that $j(A) 
\cap 
V_{g(\kappa)} = j_U(A) \cap 
V_{g(\kappa)}$.
\qed
\end{proof}

\section{There are no excessively hypercompact cardinals.} 
\label{section.ExcessivelyHC}

In 
definition 1.2 of 
\cite{Apter2011:SomeApplicationsMethod(Hypercompact)},
Apter defined an excessively hypercompact cardinal as follows.\footnote{At 
that 
time, Apter called these cardinals hypercompact rather than excessively 
hypercompact. 
But in light of theorem \ref{theorem.StrongHypercompactDNE}, we now call them 
excessively hypercompact.}

\begin{definition}
[Apter, \cite{Apter2011:SomeApplicationsMethod(Hypercompact)}]  
\label{definition.StronglyHC}
A cardinal $\kappa$ is excessively 
$0$-hypercompact iff $\kappa$ is
supercompact. For $\alpha>0$, a cardinal  $\kappa$ is excessively 
$\alpha$-hypercompact iff 
for 
any 
cardinal $\delta \geq 
\kappa$, 
there is an elementary embedding 
$j: V \to M$ witnessing the $\delta$-supercompactness of $\kappa$ (i.e.\ cp$(j) 
= 
\kappa, j(\kappa) > \delta,$ and $M^\delta \of M$) generated by a supercompact 
ultrafilter over $P_\kappa(\delta)$ such that $M \satisfies$ ``$\kappa$  is 
excessively
 $\beta$-hypercompact for every $\beta<\alpha$''. A cardinal $\kappa$ is 
\textbf{ 
excessively hypercompact} iff $\kappa$ is excessively $\alpha$-hypercompact for 
every 
ordinal 
$\alpha$.
\end{definition}

Postulating the existence of an excessively hypercompact cardinal leads to a 
contradiction.

\begin{theorem} \label{theorem.StrongHypercompactDNE}
There are no excessively hypercompact cardinals. In particular, there 
is no cardinal $\kappa$ such that $\kappa$ is excessively 
$(2^\kappa)^+$-hypercompact.
\end{theorem}

\begin{proof}
Suppose towards a contradiction that $\kappa$ is least such that $\kappa$ is 
excessively 
$(2^\kappa)^+$-hypercompact.  Apply the 
definition of excessive hypercompactness in the case  $\delta=\kappa$ to obtain 
an elementary embedding $j:V \to M$ which is witnessed by a normal fine measure 
on $P_\kappa \kappa$ (which is isomorphic to a normal measure on $\kappa$)
 such 
that $\kappa$ is excessively $\alpha$-hypercompact in $M$ for 
all $\alpha < \big((2^\kappa)^+\big)^V$. This includes all $\alpha<j(\kappa),$ 
since $j(\kappa)$ 
has cardinality $2^\kappa$ in $V$.  In particular, it includes the case of 
$\alpha = \big((2^\kappa)^+\big)^M$ , since this is less than $j(\kappa).$ By 
reflection, 
there are many $\gamma < \kappa$ such that $\gamma$ is 
excessively $(2^\gamma)^+$-hypercompact, and this contradicts the minimality of 
$\kappa.$
\qed
\end{proof}

In definition \ref{definition.hypercompact}, I describe a hypercompact 
cardinal. 
Apter had erroneously believed that this definition was equivalent to the 
definition 
of an excessively hypercompact cardinal.\footnote{personal communication with 
Apter, 2012} However, the existence 
of a hypercompact cardinal is strictly weaker in consistency strength than the 
existence of a Woodin-for-supercompactness cardinal; I prove this fact in 
theorem \ref{theorem.WSC>WHC}.
The proofs in 
\cite{Apter2011:SomeApplicationsMethod(Hypercompact)} all work using 
hypercompact cardinals in place of excessively hypercompact cardinals, so the 
error in the definition given in 
that paper did not have 
severe consequences.\footnote{personal communication with Apter, 2012.} 

\begin{definition} \label{definition.hypercompact}
The hypercompact cardinals are defined recursively as follows. Given any 
ordinal 
$\alpha$, the cardinal $\kappa$ is $\alpha$-hypercompact if and only if for 
every ordinal $\beta < \alpha$ and for every cardinal 
$\lambda \geq \kappa$, there exists a cardinal $\lambda' \geq \lambda$
and there exists an elementary embedding $j: V \to M$ generated by a normal 
fine 
measure on $P_\kappa \lambda'$ such that the 
cardinal $\kappa$ is
 $\beta$-hypercompact in $M$.  (In particular, every cardinal is 
 $0$-hypercompact, and $1$-hypercompact is equivalent to supercompact.) The 
 cardinal
$\kappa$ is \textbf{hypercompact} if and only if it is 
$\beta$-hypercompact for every ordinal $\beta$.
\end{definition}

The key difference between the definitions of hypercompact and excessively 
hypercompact is that in the definition of hypercompact, the embedding $j$ need 
not be witnessed by a normal fine measure on $P_\kappa \lambda$, but can 
be witnessed instead by a larger supercompactness measure.\footnote{An 
additional minor difference is that the definition of hypercompact 
handles limit stages differently from the definition of excessively 
hypercompact. I 
made this change in order to unify the definition for the successor and limit 
stages, and also to define hypercompact cardinals analogously to the Mitchell 
order.}

Note that both the hypercompact cardinals and the excessively hypercompact 
cardinals are first-order definable in \ZFC. Formally, the definition of a 
hypercompact cardinal is by 
recursion on $\kappa$ as follows. Assuming recursively that the set 
$$\text{HC}_{<\kappa} := \set{(\alpha, \eta) \st \eta \text{ is } 
\alpha\text{-hypercompact and } \eta < \kappa}$$ is already 
defined, define that
$\kappa$ is $\alpha$-hypercompact if and only if for every $\beta < \alpha$ and 
for every $\lambda \geq \kappa$ there exists $\lambda' \geq \lambda$ and there 
exists an elementary embedding $j: V \to M$ generated by a normal fine measure 
on $P_\kappa \lambda'$ such that $(\beta, \kappa) \in j(H_{<\kappa})$. This in 
turn can be stated formally as a first-order proposition using the Lo\'{s}
theorem, without referring explicitly to the embedding $j$. 

I now establish the consistency of a hypercompact cardinal relative to a 
Woodin-for-supercompactness cardinal. The bold part of the proof emphasizes why 
the 
proof would not work to establish the consistency of an excessively 
hypercompact 
cardinal.

\begin{theorem}\label{theorem.WSC>WHC}
If the cardinal $\delta$ is Woodin for 
supercompactness, then in the model $V_\delta$, there is a 
proper class of hypercompact cardinals.
\end{theorem}

\begin{proof}
Suppose $\delta$ is Woodin for supercompactness. 
Suppose towards a contradiction that the hypercompact 
cardinals of $V_\delta$ are bounded above by some 
cardinal $\eta$. Let the function $f: \delta \to 
\delta$ be the failure-of-hypercompactness function as defined in the model 
$V_\delta$. That is to 
say, for an ordinal $\xi< \delta$, let $f(\xi)$ be the least ordinal $\beta$ 
such 
that $\xi$ is not $\beta$-hypercompact in $V_\delta$ 
if such a $\beta$ exists, and let $f(\xi) = 0$ otherwise.

By theorem \ref{theorem.WSCcharacterization}, there 
is a
$(<\delta, f)$-supercompact cardinal $\kappa$ above $\eta$, and this fact is 
witnessed 
by a collection of 
elementary embeddings
$j_\gamma : V \to M_\gamma$ for $\gamma < \delta$. (The subscripted 
$\gamma$ serves to index the target model, not to 
refer to a rank-initial cut thereof. ) \textbf{If 
$\gamma$ is taken to be sufficiently large,} then 
$(\kappa, f(\kappa)) \in j_\gamma(f)$, and so 
$j_\gamma(f)(\kappa) = f(\kappa)$. That is to say, in 
$M_\gamma$, the cardinal $\kappa$  is
$\beta$-hypercompact for every $\beta < f(\kappa)$. By 
taking a factor embedding if necessary, assume that 
$j_\gamma$ is generated by a normal fine measure $U$ 
on $P_\kappa \gamma$ such that $U \in V_\delta$. Thus 
in $V_\delta$, the collection of embeddings $\<j_\gamma>$ witness
that $\kappa$ is $f(\kappa)$-hypercompact, 
contradicting the definition of $f$. This 
contradiction completes the proof. 
\qed
\end{proof}

Finally, I consider the extent to which the hierarchy of 
$\beta$-hypercompactness 
and the hierarchy of
excessive 
$\beta$-hypercompactness coincide for particular small values of 
$\beta$.

\begin{theorem}\label{theorem.StrongVsWeakHypercompact}
Let $\kappa$ be a cardinal, and let $\beta \leq \kappa^+$ be an ordinal. If 
$\kappa$  is
$\beta$-hypercompact, then for every ordinal $\alpha < \beta$ and for every 
cardinal $\lambda 
\geq \kappa$, there is an elementary embedding $j: V \to M$ generated by a 
normal fine measure on $P_\kappa \lambda$ such that $\kappa$ is 
$\alpha$-hypercompact in $M$. Thus, the $\beta$-hypercompactness and excessive 
$\beta$-hypercompactness hierarchies align below $\kappa^+$.\footnote{Actually, 
this alignment is off by one, because the definitions of these hierarchies 
handle limit stages differently. But this fact is a technical detail not 
germane 
to the main idea. }
\end{theorem}

\begin{proof}
The proof is by induction on ordinals $\beta$.
Suppose that the cardinal $\kappa$ is $\beta$-hypercompact and that 
the theorem is true for all $\beta' < \beta$. Let $\lambda \geq \kappa$ be a 
cardinal, and let $\alpha < \beta$. It suffices to show that there is an 
elementary embedding $j: V \to M$ generated by a normal fine 
measure in $V$ on $P_\kappa \lambda$ such that in $M$, the 
cardinal $\kappa$ is $\alpha$-hypercompact.

By hypothesis, the cardinal $\kappa$ is $\beta$-hypercompact. So for some 
cardinal $\theta \geq \lambda$, there exists an elementary 
embedding $j: V \to M$ such that in $M$, the cardinal $\kappa$ 
is $\alpha$-hypercompact.

Let $j_\lambda : V \to M_\lambda$ be the 
$\lambda$-supercompactness factor embedding induced via $j$ by the seed $j 
\image \lambda$, and 
let $k: M_\lambda \to M$ be the elementary embedding such that 
$k \circ j_\lambda = j$, as in the following commutative 
diagram.  To be precise, the embedding $j_\lambda$ is the ultrapower generated 
by $U_\lambda$, where $U_\lambda$ is the normal fine measure on $P_\kappa 
\lambda$ given by $A \in U \iff j \image \lambda \in j(A)$. (The subscript 
$\lambda$ in $M_\lambda$ serves to index the model $M_\lambda$, not to denote a 
level of its cumulative hierarchy.)
\begin{diagram}[w=5em]
V & \rTo^j & M  \\
\dTo<{j_\lambda} & \ruTo>{k} & \\
 M_\lambda & & 
\end{diagram}
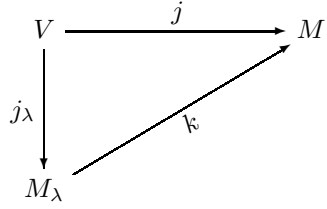
\figlist{Factor embeddings of a hypercompactness embedding}

If $M_{\lambda} = M$, then the existence of the embedding $j_\lambda$ suffices 
to complete the proof. If 
$M_\lambda \neq M$, then the elementary embedding $k$ must be 
nontrivial, and its critical point must be an greater than $\kappa$ and 
inaccessible in $M_\lambda$. The model $M_\lambda$ agrees with $V$ on 
$\kappa^+$, so this critical point must be greater than $\kappa^+$. Therefore, 
it follows from the elementarity of $k$ that $\kappa$ is 
$\alpha$-hypercompact in $M_\lambda$.
\qed
\end{proof}

\section{Enhanced supercompact cardinals} \label{section.EnhancedSC}
In this brief section, I analyze the consistency strength of an enhanced 
supercompact cardinal.  
The definition of an enhanced supercompact 
cardinal comes 
from Apter's paper, \cite{Apter2008:Reducing(EnhancedSC)}.

\begin{definition} A cardinal $\kappa$ is \textbf{enhanced 
supercompact} if 
and only 
if there exists a strong cardinal $\theta > \kappa$ 
such that for 
every cardinal $\lambda > \theta$, there exists a 
$\lambda$-supercompactness 
embedding $j: V \to M$ such that 
$\theta$ is strong in $M$.
\end{definition}

Apter required that the 
embedding $j$ be 
generated by a normal fine measure on $P_\kappa 
\lambda,$. This requirement provides a first-order characterization, but it 
adds 
no strength, because one can take a 
factor 
embedding.

This next theorem shows that a Woodin-for-supercompactness cardinal is strictly 
stronger in consistency than an enhanced supercompact cardinal.

\begin{theorem}\label{proposition.WoodinSC>ESC}
\label{theorem.WoodinSC>ESC}
Suppose the cardinal $\delta$ is Woodin for supercompactness. 
Then there are 
unboundedly many cardinals $\kappa < \delta$ such that 
$\kappa$ is a limit of 
cardinals $\eta$ such that there exists an inaccessible cardinal $\beta$ such 
that $\eta < 
\beta < \kappa$,  and
$$V_\beta \satisfies \eta \text{ is enhanced 
supercompact. }$$
\end{theorem}

\begin{proof}
The proof follows the same general line of reasoning 
as theorem 5 of 
\cite{Apter2008:Reducing(EnhancedSC)}. 
Suppose $\delta$ is Woodin for supercompactness. Let 
$f: \delta \to 
\delta$ be given by taking $f(\alpha)$ to be the 
second strong 
cardinal of $V_\delta$ greater than $\alpha$. This 
function is 
well-defined, 
since the strong cardinals of $V_\delta$ are 
unbounded, since $\delta$ is Woodin.

Let $\kappa$ be a closure point of $f$, and let $j: 
V \to M$ be an 
elementary embedding such that $M^{j(f)(\kappa)} \of M$ and $j(f)(\kappa)< 
\delta$, 
i.e.\ the embedding $j$ witnesses that $\delta$ is Woodin for 
supercompactness with respect to the function $f$. By theorem 
\ref{theorem.WSCcharacterization}, assume without loss of generality that the 
embedding $j$ is generated by a normal fine measure on $P_\kappa \lambda$ for 
some cardinal $\lambda < \delta$. It follows that $j(\delta) = \delta$.
By the definition of $f$ and the elementarity of $j$, there is 
a cardinal $\kappa_0$ such that $\kappa < \kappa_0 < 
j(f)(\kappa)$, 
and the cardinal $\kappa_0$ is strong in the model 
$M_{j(\delta)} = M_\delta$, and 
furthermore, the cardinal
$j(f)(\kappa)$ is strong in the model $M_{\delta}$.\footnote{Actually, it 
suffices 
for the proof that $j(f)(\kappa)$ is inaccessible.}

For each cardinal $\lambda$ such that $\kappa_0 < \lambda < 
j(f)(\kappa)$, let $U_\lambda$ be the normal fine measure on $P_\kappa \lambda$ 
given by $A \in U \iff j \image \lambda \in j(A)$. Let $j_\lambda: V \to 
M_\lambda$ be the $\lambda$-supercompactness 
embedding 
generated by $U_\lambda$, and let $i: M_\lambda \to M$ be the 
elementary 
embedding such that $i \circ j_\lambda = j$. (The subscripted $\lambda$ serves 
to index the model $M_\lambda$, not to denote a level of its cumulative 
hierarchy.)
\begin{diagram}[w=5em]
V & \rTo^j & M  \\
\dTo<{j_\lambda} & \ruTo>{i} & \\
 M_\lambda & & 
\end{diagram}
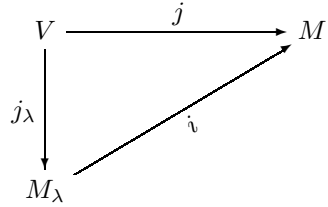
\figlist{Factor embeddings of a Woodin-for-supercompactness 
embedding}

 Suppose 
towards a 
contradiction that for some cardinal $\gamma$ with 
$\kappa_0 < 
\gamma < 
j(f)(\kappa),$ and for some cardinal $\lambda$ such that $\kappa_0 < \lambda < 
j(f)(\kappa)$, 
$$M_\lambda \satisfies \kappa_0 \text{ is not 
$\gamma$-strong}.$$

Then by elementarity, 
$$M \satisfies i(\kappa_0) \text{ is not 
$i(\gamma)$-strong}.$$
But $i$ fixes $\kappa_0$, and so this contradicts 
the fact that the cardinal $\kappa_0$ is 
strong in $M_\delta$, since $i(\gamma) < 
i(j(f)(\kappa)) \leq 
j(j(f)(\kappa)) < \delta$. From this contradiction, I 
conclude that for all 
cardinals $\gamma$ and $\lambda$, if $\kappa_0 < \gamma < 
j(f)(\kappa)$ and $\kappa_0 < \lambda < 
j(f)(\kappa)$ then 
$$M_\lambda \satisfies \kappa_0 \text{ is 
$\gamma$-strong}.$$

Finally, from the closure of $M$, it follows that $U_\lambda \in 
M_{j(f)(\kappa)}$ for each 
cardinal 
 $\lambda$ such that $\kappa_0 < \lambda < j(f)(\kappa)$. Furthermore, for each 
 such cardinal $\lambda$, 
the 
elementary 
embedding generated by $U_\lambda$ in the model $M$
is equal to  
$j_\lambda \restrict M$. Since 
$\lambda$ was taken to be an arbitrary cardinal 
between $\kappa_0$ 
and $j(f)(\kappa)$, it follows that in the model
$M_{j(f)(\kappa)}$, the 
cardinal $\kappa$ is enhanced supercompact. 

By reflection, in $V_\kappa$, there are unboundedly 
many 
cardinals $\eta$ such that for some inaccessible cardinal $\beta$ with $\eta < 
\beta < 
\kappa$, 
$$V_\beta \satisfies \eta \text{ is enhanced supercompact. }$$

By a simple modification to the function $f$, the cardinal 
$\kappa$ can be made arbitrarily large below 
$\delta$. The conclusion of the theorem 
follows. 
\qed
\end{proof}

\section{High-jump cardinals and forcing} \label{section.forcing}

In this section, I prove some
results about the preservation and destruction of \hj\ cardinals by forcing.

Suppose $j: V \to M$ is a \hj\ embedding, and $V[G]$ is a forcing extension of 
$V$. Under what conditions does $j$ lift to a \hj\ embedding $j^* : V[G] \to 
M[H]$? The conditions under which a supercompactness embedding lifts to a 
supercompactness embedding have been well-studied in the literature. The 
following lemma extends these conditions to provide conditions for which a \hj\ 
embedding lifts to a \hj\ embedding.

\begin{lemma} \label{lemma.HJPreservation} Suppose $j: V \to M$ 
is a \hj\ embedding for
$\kappa$ with clearance $\theta$. Let $V[G]$ be a forcing
extension of $V$, and suppose that $j$ lifts
to a $\theta$-supercompactness embedding $j^*: V[G] \to M[H]$.\footnote{By a 
$\theta$-supercompactness embedding, I simply mean that $M[H]$ is sufficiently 
closed, not that the embedding is generated by a normal fine measure.} 
Let $U$ be the normal measure on $\kappa$ given by $A \in U \iff
\kappa \in j(A)$. \linebreak If the family of functions 
$(\kappa^\kappa)^V$ is $\leq_U$-unbounded in 
$(\kappa^\kappa)^{V[G]},$ then the lifted embedding $j^*$ is a \hj\ 
embedding. 
Furthermore, if $M[H]^{\theta^+} \nsubseteq M[H]$ in $V[G]$, then the 
conclusion 
can be 
strengthened to a biconditional: the lifted embedding $j^*$ is a 
\hj\ embedding if and only if the family of functions $(\kappa^\kappa)^V$ is 
$\leq_U$-unbounded 
in $(\kappa^\kappa)^{V[G]}.$ \end{lemma}

\begin{proof}
Note that since $U$ is an ultrafilter, the family of functions 
$(\kappa^\kappa)^V$ is $\leq_U$-unbounded in $(\kappa^\kappa)^{V[G]}$ if and 
only this family is a dominating family, which is true if and only if the 
forcing does not add a $U$-dominating function. 

To prove the first part of the theorem, assume that the family of functions 
$(\kappa^\kappa)^V$ is 
$\leq_U$-unbounded in $(\kappa^\kappa)^{V[G]}.$ In $V[G]$, let $f: \kappa \to 
\kappa$. It suffices
to show that $j^*(f)(\kappa) < \theta$. Since $(\kappa^\kappa)^V$ is a 
dominating family, there is a function $g \in (\kappa^\kappa)^V$ such 
that $f \leq_U g$. It 
follows that
$$j(f)(\kappa) \leq j(g)(\kappa) < \theta,$$ and so the lifted embedding is a 
\hj\ embedding.

To prove the second part of the theorem, suppose that $M[H]^{\theta+} 
\nsubseteq 
M[H]$ in $V[G]$, and that $(\kappa^\kappa)^V$ is $\leq_U$-bounded by some 
function $g \in (\kappa^\kappa)^{V[G]}.$ Then $j^*(g)(\kappa) \geq 
j^*(f)(\kappa)$ 
for every $f \in (\kappa^\kappa)^V$, and so in particular $j^*(g)(\kappa) \geq 
\theta$, and so the function $g$ witnesses that $j^*$ is not a \hj\ embedding.
\qed
\end{proof}

One particular important instance where the class of functions 
$(\kappa^\kappa)^V$ is unbounded 
in 
$(\kappa^\kappa)^{V[G]}$ is if the forcing satisfies the $\kappa$-chain 
condition.

The biconditional version of the lemma actually holds even holds in many cases 
where $M[H]$ is closed 
under sequences of length greater than $\theta$ --- given a $g$ such that 
$j^*(g)(\kappa) \geq \theta$, one can easily modify the function $g$ to produce 
another function $h$ such that $j^*(h)(\kappa)$ is much larger than $\theta$. 
For instance, let $h(\alpha)$ be the least measurable cardinal above 
$g(\alpha)$, so that $j^*(h)(\kappa)$ is the least measurable 
cardinal of $M[H]$ above $\theta$.

The next theorem addresses the preservation of \hj\ cardinals in the downwards 
direction.

\begin{theorem}\label{theorem.HJapproxcover}
Suppose $V \of \overline{V}$ satisfies the $\delta$ approximation 
and cover properties, and for some cardinals $\kappa, \theta > \delta$ 
there is a \hj\ measure $U$ on $P_\kappa {\theta}$ in $\overline{V}$. 
Then there is a \hj\ measure on $P_\kappa {\theta}$ in $V$ as 
well.
\end{theorem}

\begin{proof}
Let $j: \overline{V} \to \overline{N}$ be the elementary embedding 
generated by $U$ in $\overline{V}$.
By the proof of corollary 26 of 
\cite{Hamkins2003:ExtensionsWithApproximationAndCoverProperties}, the 
restricted 
embedding $j \restrict V: V \to N$ is amenable 
with $V$, and $N^\theta \of N$ in $V$. In particular, $j \restrict V$ is a \hj\ 
embedding. Let $j_0: V \to M$ be 
the $\theta$-supercompactness factor embedding induced via $j\restrict V$ by 
the 
seed $j 
\image \theta$. 
Let $f: \kappa \to \kappa$ be a function.  It follows 
from
lemma \ref{lemma.FactorEmbeddingJump} applied in $V$ to the embedding $j 
\restrict V$ that $j_0$ is a \hj\ embedding. Furthermore, the factor 
embedding construction ensures that $j_0$ is generated by a 
measure that is an element of $V$, so the proof is complete.
\qed
\end{proof}

Next, I show that the previous two results together prove the analogue of the 
Levy-Solovay theorem for \hj\ cardinals.
\begin{theorem}\label{theorem.HJPreservation}
Let $\P$ be a forcing 
notion such that $|\P| < \kappa$. Let $G \of \P$ be $V$-generic. 
Then in $V[G]$, the cardinal $\kappa$ is \hjp\ if and only if $\kappa$ is \hjp\ 
in $V$. 
\end{theorem}

\begin{proof}
Since the forcing is small, in particular it satisfies the $\kappa$-chain 
condition, and so every function $f: \kappa \to \kappa$ 
in $V[G]$ is bounded by such a function in $V$. Thus, the upwards direction of 
the proof follows from lemma \ref{lemma.HJPreservation}. By lemma 13 of 
\cite{Hamkins2003:ExtensionsWithApproximationAndCoverProperties}, the forcing 
$\P$ satisfies the $\delta$ approximation and cover properties for some 
cardinal 
$\delta < \kappa$. Thus, the downwards 
direction of the 
proof follows immediately from theorem \ref{theorem.HJapproxcover}.
\qed
\end{proof}

Next, I apply lemma \ref{lemma.HJPreservation} to show that the canonical 
forcing of the \GCH\ preserves \hj\ cardinals.

\begin{theorem} \label{theorem.hjGCH}
Every \hj\ cardinal is preserved by the canonical forcing $\P$ of the \GCH.  To 
be precise, the forcing $\P$ is defined as the 
 Easton support product over all infinite cardinals 
$\delta$  of $\Add(\delta^+, 1)$. 
\end{theorem}

\begin{proof}
Let $G \of \P$ be $V$-generic.
In $V$, let $U$ be a \hj\ measure on $P_\kappa \theta$ for some cardinals 
$\kappa$ and $\theta$. Let $j_U$ be the \hj\ embedding generated by $U$. 
It follows from a standard argument that the embedding $j_U$ lifts to a 
$\theta$-supercompactness embedding $j_U^*: V[G] \to M[H]$.\footnote{For 
details, see theorem 105 of Hamkins's unpublished book, \emph{Forcing and Large 
Cardinals}.}  To complete the 
proof that \hj\ cardinals are preserved by $\P$,  it suffices to show that 
every function on $f: \kappa \to \kappa$ in $V[G]$ is dominated by such a 
function in $V$ and 
then apply lemma \ref{lemma.HJPreservation}. 
Towards this end, note that the forcing 
$\P$ factors as 
$\P_{<\!\kappa}\ast \P_{\geq \kappa}$. The first factor 
satisfies the $\kappa$-chain condition, and so every function
$f: \kappa \to \kappa$ added by it is dominated by a ground model function. The 
second factor is 
$\leq\!\kappa$-closed, and so it adds no new function $f: 
\kappa \to \kappa$.
\qed
\end{proof}

High-jump cardinals are in general much more fragile than supercompact 
cardinals, as is shown by the following theorem.

\begin{theorem} \label{theorem.HJFragility}
Let $\kappa$ be a \hj\ cardinal. After forcing with $\Add(\kappa, 1)$ or 
with
$Add(\kappa^+, 1)$, the cardinal $\kappa$ is no longer a \hj\ cardinal.
\end{theorem}
\begin{proof}
A recent theorem of Bagaria, Hamkins, and Tsaprounis shows that superstrong 
cardinals are destroyed by these forcings, among others 
\cite{BagariaHamkinsTsaprounis:SuperstrongFragility}. 
By corollary \ref{corollary.SuperstrongFactorofHJ}, every \hj\ cardinal is also 
superstrong, so it follows that these forcings also destroy \hj\ cardinals.
\qed
\end{proof}

Finally, I show that if the cardinal $\kappa$ is \hjp, then there is a forcing 
extension where $\kappa$ is still \hjp\ but is not supercompact.

\begin{theorem} \label{theorem.HJButNotSC}
Suppose there exists a \hj\ measure on $P_\kappa \theta$, and furthermore, the 
cardinal $\kappa$ is supercompact. Let $\P$ be any 
forcing smaller than $\kappa$. Let $g \of \P$ be $V$-generic. Let $\Q$ be any 
nontrivial forcing that is $\leq 
\!\!\scexp{\theta}{\kappa}$-closed in $V[g]$, and let $G \of \Q$ be 
$V[g]$-generic. Then in $V[g][G]$, there is still a \hj\ measure on $P_\kappa 
\theta$, but the cardinal
$\kappa$ is not supercompact.
\end{theorem}

\begin{proof}
Since the forcing $\P$ is small relative to $\kappa$, the cardinal $\kappa$ is 
still both supercompact and \hjp\ in $V[g]$. Because of the closure condition 
on 
the forcing $\Q$, this forcing does not add any subsets or elements to 
$P_\kappa 
\theta$, nor does it add any new functions $f: \kappa \to \kappa$. Therefore, 
in 
$V[g][G]$, the cardinal $\kappa$ is still \hjp. However, since the forcing $\Q$ 
is nontrivial, there is a cardinal $\lambda > \scexp{\theta}{\kappa}$ such that 
$\Q$ adds a subset to $\lambda$. By a theorem of Hamkins 
and Shelah (\cite[p.551]{HamkinsShelah98:Dual}), the cardinal $\kappa$ is no 
longer $\lambda$-supercompact in $V[g][G]$.
\qed
\end{proof}

Some open questions on the topics of this section are as follows. 

\begin{question} \label{question.HJcardinalFragility}
Suppose $j: V \to M$ is a \hj\ embedding for $\kappa$ with clearance $\theta$. 
What types of forcing, if any, preserve the $\theta$-supercompactness of 
$\kappa$ while destroying the \hj\ cardinal property of $\kappa$?
\end{question}

Question \ref{question.HJcardinalFragility} can be further refined to a 
question 
about individual embeddings rather than about cardinals.

\begin{question} \label{question.HJEmbeddingfragility}
Let $j: V \to M$ be a \hj\ embedding for $\kappa$ generated by a \hj\ measure. 
Let $\P$ be a forcing notion, and suppose that the embedding $j$ lifts over 
$\P$ 
such that the lift is a supercompactness embedding. Under what conditions does 
the lift fail to be a \hj\ embedding?
\end{question}

\section{Laver functions for \hj\ cardinals} \label{section.Laver}
In this section, I define Laver functions for \shj\
cardinals and establish their existence under suitably strong hypotheses. 
Laver functions were originally defined for supercompact cardinals in 
\cite{Laver78}.

Given a supercompact cardinal $\kappa$, a supercompactness Laver function for 
$\kappa$ is a partial function $\ell\from \kappa \to V_\kappa$ such that for 
every 
cardinal 
$\lambda$ and for every set $x \in H_{\lambda^+}$ there is a 
$\lambda$-supercompactness embedding generated by a normal fine measure on 
$P_\kappa 
\lambda$ such that $j(\ell)(\kappa) = x$.  One can also put additional 
requirements on the domain of a supercompactness Laver function. For 
instance, one can require that each $\gamma \in \dom(\ell)$ is an inaccessible 
cardinal such that the closure property $\ell \image \gamma \of V_\gamma$ is 
satisfied.
A \shj\ Laver function is defined similarly to a supercompactness Laver 
function, as follows.
\begin{definition}
Given a \shj\ cardinal $\kappa$, a \textbf{\shj\ Laver function} for 
$\kappa$ is a partial function $\ell \from \kappa \to V_\kappa$ satisfying the 
following properties. For 
every 
set $x$, 
for unboundedly many cardinals $\delta$, there is \hj\ embedding with critical 
point $\kappa$ and clearance $\delta$, generated by a \hj\ measure, such that 
$j(\ell)(\kappa) = x$. Furthermore, for every ordinal $\gamma \in \dom(\ell)$, 
the closure property $\ell \image \gamma \of V_\gamma$ holds.
\end{definition}

I will prove that every supercompact cardinal has a 
supercompactness
Laver 
function anticipating every set, and whenever the cardinal $\kappa$ is 
$2^{\theta^{<\kappa}}$-supercompact, there is a supercompactness Laver 
function for $\kappa$ anticipating every set in $H_\theta^+$.

The analysis for \hj\ cardinals is more complicated than in the case of 
supercompact cardinals, 
because a supercompactness factor embedding of a \hj\ embedding is not in 
general a \hj\ embedding. For this reason, the \hj\ cardinals with excess 
closure are a useful 
tool.
As a warm-up exercise before reading the proof of the existence of \shj\ Laver 
functions, the reader may wish to review proposition \ref{proposition.HJplus}, 
which uses a related technique. 




\begin{theorem}\label{theorem.SHJLaverFunction}
Let $\kappa$ be a cardinal. Then there exists a partial function $\ell \from 
\kappa \to V_\kappa$ such that for all cardinals $\theta$, if there is a \hj\ 
measure on $P_\kappa 2^\theta$ generating an ultrapower embedding with 
clearance 
$\theta$, then in the model $V_\theta$, the function $\ell$ is a \shj\ 
Laver function for $\kappa$. 
\end{theorem}

\begin{proof}
Define the function $\ell$ recursively as follows. Suppose that $\ell 
\restrict \gamma$ has been defined. Define $\ell(\gamma)$ as 
described in the next paragraph if the relevant hypotheses 
hold. Otherwise, leave $\gamma$ out of the domain of $\ell$. 

Suppose that $\ell \image \gamma \of V_\gamma$ and that furthermore, in the 
model $V_\kappa$, some set $x$  witnesses that the function $\ell \restrict 
\gamma$ is not a 
\shj\ Laver function for $\gamma$. \emph{That is to say, in the 
model 
$V_\kappa$, there is a cardinal $\delta_0$ such that for all cardinals $\delta 
> 
\delta_0$, there is no
elementary embedding $j: V \to M$ with critical point $\gamma$ and clearance 
$\delta$, generated by a \hj\ 
measure, such that $j(\ell\restrict \gamma)(\gamma) = x$.} Then pick a set $x 
\in 
V_\kappa$ of minimal $\in$-rank among all sets with this property, and let
$\ell(\gamma) = x$. 

I now verify that the function $\ell$ has the desired feature. Suppose not. 
Then 
there is 
some cardinal $\theta$ such that there is a \hj\ measure $\mu$ on $P_\kappa 
2^\theta$ 
generating an ultrapower embedding with clearance $\theta$, 
but in the model $V_\theta$, some set $x$ witnesses that the function $\ell$ 
fails to be a \shj\ Laver function for $\kappa$. Let
$j: V \to M$ be the ultrapower generated by $\mu$. By 
lemma \ref{lemma.VThetaElemMjkappa}, the 
elementarity relation $V_\theta \elem M_{j(\kappa)}$ holds. Therefore, in the 
model $M_{j(\kappa)}$, the set $x$ witnesses that the function $\ell$ is not a 
\shj\ Laver 
function for $\kappa$. That is to say, in the model $M_{j(\kappa)}$, there is 
some cardinal $\delta_0$ such that for all cardinals  $\delta 
> \delta_0$, there does not exist a \hj\
embedding $h$, generated by a \hj\ measure, with critical point $\kappa$ and 
clearance 
$\delta$, such that 
$h(\ell)(\kappa) = x$.

Accordingly, since $j(\ell) 
\restrict \kappa = \ell$, it follows from the definition of the function $\ell$ 
and from the 
elementarity of the embedding $j$ that $j(\ell)(\kappa)$ is 
defined and equal to some set $y \in M_{j(\kappa)}$ such that in the model 
$M_{j(\kappa)}$, 
the set $y$ witnesses that the function $\ell$ is not a \shj\ Laver function 
for 
$\kappa$.
Furthermore, this set $y$ is of minimal $\in$-rank, and so $y \in V_\theta$, 
since $V_\theta \elem M_{j(\kappa)}$.

Let $U$ be the 
$\theta$-supercompactness measure on $P_\kappa \theta$ induced by $j$ via the 
seed $j \image 
\theta$. Let $j_U: V \to N$ be the supercompactness embedding generated by $U$, 
and let the elementary embedding $k$ be such that the diagram below 
commutes.

\begin{diagram}[w=5em,h=30pt]
 V & \rTo^j
& M   \\
\dTo<{j_U} & \ruTo>{k} & \\
 N & & 
\end{diagram}
\figlist{Factor embeddings of a \hj\ embedding with excess closure} 

 By reasoning similar to the proof of lemma 
\ref{lemma.HJFactor}, the measure $U$ is a \hj\ measure, the clearance 
of $j_U$ is $\theta$, and
$j_U(\ell)(\kappa) = y$. Furthermore, since the model $M_{j(\kappa)}$ is closed 
in $V$ under 
sequences of length $2^{\theta}$, and since $\theta^\kappa = \theta$ by 
lemma \ref{lemma.ClearanceFactsAHJEmbedding}, it follows that $U \in 
M_{j(\kappa)}$, and the model $M_{j(\kappa)}$ 
correctly computes that $j_U(\ell)(\kappa) = y$. (In the previous sentence, 
$j_U$ denotes the 
embedding generated by the measure $U$ in the model $M_{j(\kappa)}$.) However, 
$\theta > \delta_0$, 
so 
this computation contradicts the fact that $y$ is not anticipated in 
$M_{j(\kappa)}$ by 
$\ell$ with respect to any \hj\ 
embedding with clearance greater than $\delta_0$ that is generated by a \hj\ 
measure.
\qed
\end{proof}

\begin{corollary} \label{corollary.SHJLaverFunction} 
Suppose that for some cardinal $\kappa$, there is an unbounded set of cardinals 
$\theta$ such that there is a \hj\ measure on $P_\kappa 2^\theta$ generating an 
ultrapower embedding with clearance $\theta$.
Then there exists a \shj\ Laver function for $\kappa$ in $V$.
\end{corollary}
\begin{proof}
Define the function $\ell$ as in theorem \ref{theorem.SHJLaverFunction}. It 
follows from theorem \ref{theorem.SHJLaverFunction} 
that for arbitrarily large cardinals $\theta$, the function $\ell$ is a \shj\ 
Laver function for $V_\theta$. Furthermore, for such a cardinal $\theta$, if $U 
\in V_\theta$ is a \hj\ measure, generating an embedding $j_U: V \to M$, then 
the embedding generated by the measure $U$ in the model $V_\theta$ is $j_U 
\restrict V_\theta$. It follows that $\ell$ is a \shj\ Laver function for 
$\kappa$.
\qed
\end{proof}

%
%
%
%


It is also possible to define Laver functions for other large cardinal notions 
related to \hj\ cardinals, for instance for particular types of \hj\ cardinals 
with excess closure. For an example, see \cite[theorem 
116]{Perlmutter2013:Dissertation}.

I close the section with a question.

\begin{question}
Is it possible to prove the existence or the consistency of a \shj\ Laver 
function from a hypothesis substantially weaker than that of theorem 
\ref{theorem.SHJLaverFunction}?
\end{question}

For instance, it may be possible to force the existence of a \shj\ Laver 
function for $\kappa$, beginning in a model where $\kappa$ is only \shj.


\section{Ideas for further research}\label{section.FurtherDefinitionIdeas} 
\label{Section.FurtherStudy}
In this section, I review some of the areas for further research discussed in 
previous sections, and I also suggest a few additional areas for further 
research.

One relationship between cardinals in the chart is 
unresolved. I do not know the relationship between enhanced 
supercompact cardinals and hypercompact cardinals. 
One established large cardinal is conspicuously missing from my 
analysis. An extendible cardinal is known to be intermediate in 
consistency strength between a supercompact cardinal and a 
\Vopenka\ cardinal. But I don't know the relationship between an 
extendible cardinal and a hypercompact cardinal or an 
enhanced supercompact cardinal. Furthermore, the $C^{(n)}$-extendible 
cardinals, 
introduced by Bagaria in \cite{Bagaria2012:CnCardinals}, fall into the large 
cardinal hierarchy between an extendible cardinal and a \Vopenka\ cardinal.

Another possible direction for further research would be to define more large 
cardinal 
notions by modifying the definitions that I have already given. One
possibility would be to modify the definition of a \hj\ cardinal so that $M$ is 
closed under $j(f)(\kappa)$-sequences for all $f: \kappa \to ORD$ rather than 
just for $f: \kappa \to \kappa$. Such a cardinal would be huge, to say the 
least. 
This hugeness is witnessed by the case where $f$ is the function with constant 
value 
$\kappa$.


As discussed in the conclusion of section \ref{section.forcing},  more work can 
also 
be done on the relationships between \hj\ cardinals and forcing. More 
generally, 
there 
is more work to 
be 
done proving forcing results for all of the cardinals discussed in this paper, 
and in particular, many of the results from section \ref{section.forcing} could 
be extended to apply to other large cardinals in this paper. 

Keeping in mind that many of the large cardinals in this paper were first 
applied towards 
universal indestructibility results, it is an interesting goal to use the new 
large cardinals to weaken the hypotheses for universal indestructibility 
results, with the goal of eventually proving an equiconsistency between a 
universal indestructibility result and a large cardinal notion. One could also 
prove new universal indestructibility results. More generally, the large 
cardinals that I have studied here could be used to weaken the 
hypotheses or find equiconsistencies for other set-theoretic results as well.

\section{Acknowledgments}
I would like to thank my dissertation advisor, Joel David Hamkins, for his 
advice relating to this research. This paper uses Paul Taylor's Commutative 
Diagrams in TeX package.

%


\bibliographystyle{plain}
\bibliography{NormansBiblio,WorksByNorman}
\end{document}

%% file: MathMacrosJDHwithoutTheorems.tex
\newcommand{\QED}{\end{proof}}

\def\BF#1.{{\bf #1.}}

\renewcommand{\P}{{\mathbb P}}
\newcommand{\Q}{{\mathbb Q}}

\newcommand{\from}{\mathbin{\vbox{\baselineskip=2pt\lineskiplimit=0pt
                         \hbox{.}\hbox{.}\hbox{.}}}} 
\newcommand{\of}{\subseteq}

\newcommand{\set}[1]{\{\,{#1}\,\}}
\newcommand{\singleton}[1]{\left\{{#1}\right\}}

\newcommand{\dom}{\mathop{\rm dom}}

\newcommand{\cof}{\mathop{\rm cof}}

\newcommand{\Add}{\mathop{\rm Add}}

\newcommand{\Con}{\mathop{{\rm Con}}}
\newcommand{\image}{\mathbin{\hbox{\tt\char'42}}}

\newcommand{\restrict}{\upharpoonright} 
\newcommand{\satisfies}{\models}

\newcommand{\smalllt}{\mathrel{\mathchoice{\raise2pt\hbox{$\scriptstyle<$}}{\raise1pt\hbox{$\scriptstyle<$}}{\raise0pt\hbox{$\scriptscriptstyle<$}}{\scriptscriptstyle<}}}
\newcommand{\smallleq}{\mathrel{\mathchoice{\raise2pt\hbox{$\scriptstyle\leq$}}{\raise1pt\hbox{$\scriptstyle\leq$}}{\raise1pt\hbox{$\scriptscriptstyle\leq$}}{\scriptscriptstyle\leq}}}

\newcommand{\boolval}[1]{\mathopen{\lbrack\!\lbrack}\,#1\,\mathclose{\rbrack\!\rbrack}}
\def\[#1]{\boolval{#1}}

\newcommand{\UnderTilde}[1]{{\setbox1=\hbox{$#1$}\baselineskip=0pt\vtop{\hbox{$#1$}\hbox to\wd1{\hfil$\sim$\hfil}}}{}}
\newcommand{\Undertilde}[1]{{\setbox1=\hbox{$#1$}\baselineskip=0pt\vtop{\hbox{$#1$}\hbox to\wd1{\hfil$\scriptstyle\sim$\hfil}}}{}}
\newcommand{\undertilde}[1]{{\setbox1=\hbox{$#1$}\baselineskip=0pt\vtop{\hbox{$#1$}\hbox to\wd1{\hfil$\scriptscriptstyle\sim$\hfil}}}{}}
\newcommand{\UnderdTilde}[1]{{\setbox1=\hbox{$#1$}\baselineskip=0pt\vtop{\hbox{$#1$}\hbox to\wd1{\hfil$\approx$\hfil}}}{}}
\newcommand{\Underdtilde}[1]{{\setbox1=\hbox{$#1$}\baselineskip=0pt\vtop{\hbox{$#1$}\hbox to\wd1{\hfil\scriptsize$\approx$\hfil}}}{}}

\newcommand{\st}{\mid}

\def\<#1>{\langle#1\rangle}
\newcommand{\ot}{\mathop{\rm ot}\nolimits}

\newcommand{\ORD}{\mathop{{\rm ORD}}}

\newcommand{\ZFC}{{\rm ZFC}}

\newcommand{\GCH}{{\rm GCH}}

\newcommand{\cell}[1]{\boxit{\hbox to 17pt{\strut\hfil$#1$\hfil}}}
\newcommand{\head}[2]{\lower2pt\vbox{\hbox{\strut\footnotesize\it\hskip3pt#2}\boxit{\cell#1}}}
\newcommand{\boxit}[1]{\setbox4=\hbox{\kern2pt#1\kern2pt}\hbox{\vrule\vbox{\hrule\kern2pt\box4\kern2pt\hrule}\vrule}}
\newcommand{\Col}[3]{\hbox{\vbox{\baselineskip=0pt\parskip=0pt\cell#1\cell#2\cell#3}}}
\newcommand{\tapenames}{\raise 5pt\vbox to .7in{\hbox to .8in{\it\hfill input: \strut}\vfill\hbox to
.8in{\it\hfill scratch: \strut}\vfill\hbox to .8in{\it\hfill output: \strut}}}
\newcommand{\Head}[4]{\lower2pt\vbox{\hbox to25pt{\strut\footnotesize\it\hfill#4\hfill}\boxit{\Col#1#2#3}}}
\newcommand{\Dots}{\raise 5pt\vbox to .7in{\hbox{\ $\cdots$\strut}\vfill\hbox{\ $\cdots$\strut}\vfill\hbox{\
$\cdots$\strut}}}

%% file: MathMacrosNLP.tex
\usepackage{url} 
\usepackage{diagrams}\diagramstyle[tight,centredisplay,dpi=600, heads=LaTeX]

\DeclareMathOperator{\crit}{crit}

\newcommand{\elem}{\prec} 

\newcommand{\scexp}[2]{2^{{#1}^{< #2}} }

\newarrow{ExistsArrow} {}{dash}{}{dash}>
\newarrow{Exists} {}{dash}{dash}{dash}>
\newarrow{MapsTo} |--->

\newcommand{\Vopenka}{Vop\v{e}nka}
\newcommand{\Los}{\L o\'{s} }